\documentclass[12pt]{amsart}
\usepackage[shortlabels]{enumitem}
\usepackage{fullpage,graphicx,psfrag,amsmath,amsfonts,verbatim, amssymb, caption, subcaption,booktabs, hyperref}
\newtheorem{theorem}{Theorem}[section]
\newtheorem{prop}{Proposition}[section]
\newtheorem{corollary}{Corollary}[section]
\newtheorem{lemma}[theorem]{Lemma}
\newtheorem{problem}[theorem]{Problem}

\newcommand{\Z}{\mathbb{Z}}
\newcommand{\R}{\mathbb{R}}

\title{Dynamic dimensional reduction in the Abelian sandpile}
\author{Ahmed Bou-Rabee}
\address{Ahmed Bou-Rabee\\
  Department of Statistics\\
  University of Chicago\\
  5747 S. Ellis Avenue\\
  Chicago, IL 60637, USA}
  

\begin{document}
	
	\begin{abstract}
		We prove the dimensional reduction conjecture of Fey, Levine, and Peres (2010) on the hypercube. 
		The proof shows that dimensional reduction, symmetry, and regularity of the Abelian sandpile persist during the parallel toppling process. This stronger result verifies empirical observations first documented by Liu, Kaplan, and Gray (1990).  
	\end{abstract}
	\maketitle
	
	\section{Introduction}
	
	\subsection{Overview} 
	Let $\mathcal{C}_N^{(d)}$ be a hypercube of side length $N$ centered at the origin in the $d$-dimensional integer lattice $\Z^d$.  A {\it sandpile} is a function $s: \mathcal{C}_N^{(d)} \to \Z$. Start with $2d$ chips in the hypercube,  $s = s_0^{(d)} \equiv 2d$, then iterate the following rule: every site with at least $2d$ chips on it becomes {\it unstable} 
	and {\it topples} in {\it parallel}, simultaneously giving one chip to each of its $2d$ neighbors. If a boundary site topples, it loses chips over the edge. Eventually every site is stable and the process stops, yielding a sequence of sandpiles over time $\{s_t^{(d)}\}_{t \geq 0}$ \cite{bak1987self, dhar1990self, holroyd2008chip,levine2010sandpile,levine2017laplacian,jarai2018sandpile, klivans2018mathematics}.
	
	Simulations suggest an exact relationship between these sandpiles across size $N$, time $t$, and dimension $d$ --- see Figures \ref{fig:2d} and \ref{fig:3d}. 
	Specifically, it appears that smaller sandpiles are embedded in larger sandpiles of the same dimension at certain times. Moreover, 
	central cross sections of $d$-dimensional sandpiles coincide almost exactly with $(d-1)$-dimensional sandpiles for all time.  
	
	In this article, we provide a rigorous proof of these observations, {\it self-similarity} and {\it dimensional reduction}, 
	via a simultaneous induction on all parameters: size, time, and dimension. Along the way, we develop new techniques for analyzing the parallel toppling process which may be of independent interest.  

	\begin{figure}[!hb]
		\begin{centering}
			\includegraphics[width=0.22\textwidth]{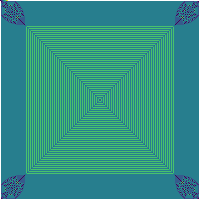}
			\includegraphics[width=0.22\textwidth]{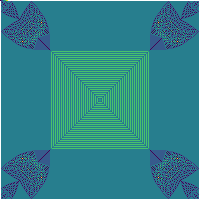} 
			\includegraphics[width=0.22\textwidth]{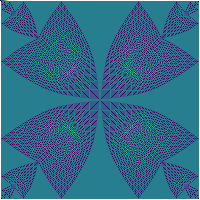}
			\includegraphics[width=0.22\textwidth]{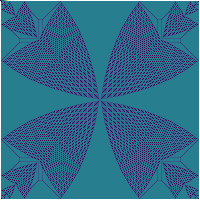}
		\end{centering}
		\caption{ $s_t^{(2)}$ for $t = (25)^2, (50)^2, (100)^2, \infty$,  where $N=200$ and $s_0 \equiv 4$. Sites with $0,\ldots,7$ chips 
			are represented by different colors.  \label{fig:2d} }
	\end{figure}
	
	\begin{figure}[tbp] 
		\begin{centering}
			\includegraphics[width=.22\textwidth]{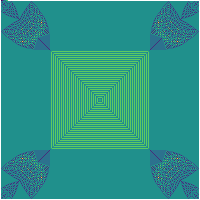}
			\includegraphics[width=.22\textwidth]{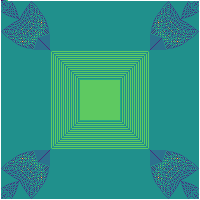}
			\includegraphics[width=.22\textwidth]{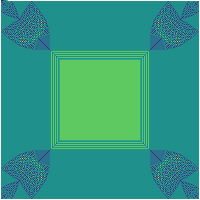}
			\includegraphics[width=.22\textwidth]{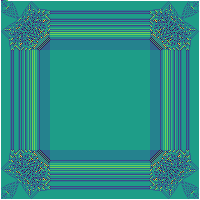}
		\end{centering}
		\setlength{\belowcaptionskip}{-20pt} 
		\caption{ $s_t^{(3)}(\cdot, \cdot, \mbox{offset})$ for 
			$\mbox{offset} = 0, 20, 40, 60$, where $t = (50)^2$, $N=200$, and $s_0 \equiv 6$.  Sites with $0,\ldots,11$ chips are represented by different colors and `offset' 
			is distance to the center of $\mathcal{C}_N^{(3)}$. \label{fig:3d} }
	\end{figure}
	Dimensional reduction in sandpiles has been known experimentally since at least the work of Liu, Kaplan, and Gray in 1990 \cite{liu1990geometry}.
	In 2010, Fey, Levine, and Peres formulated an approximate dimensional reduction conjecture in the context of the single-source sandpile \cite{fey2010growth} --- this was later highlighted in a survey by Levine and Propp \cite{levine2010sandpile}.  Soon after, Pegden and Smart, in a breakthrough, used the theory of viscosity solutions to show that the large-scale patterns which appear in sandpiles can be described by a limiting {\it sandpile PDE} \cite{pegden2013convergence}. This led to a series of works, joint with Levine, which provide a remarkable understanding of sandpiles on $\Z^2$ \cite{levine2017apollonian, levine2016apollonian, pegden2020stability}. 
	
	The structure of the sandpile PDE reveals that on torus-like domains one can construct limit 
sandpiles in $\R^d$ from those in $\R^{d-1}$.  This is also possible for certain random sandpiles \cite{bou2021convergence}. However, it has remained completely open to prove (or disprove) dimensional reduction for any natural example, on either the lattice or the continuum, till now. 

	Our proof in this paper leverages discrete techniques to establish dimensional reduction on the hypercube when $s_0 \equiv 2d$. The main insight is recognizing that the parallel toppling process together with strong induction can be used to control the sandpile as it stabilizes.  We identify a delicate interplay between the initial condition and the symmetry of the hypercube which forces the `flow' of parallel toppling to preserve dimensional reduction. In fact, we do not know how to prove dimensional reduction for only the terminal sandpile. Our proof involves no technology from viscosity solutions nor knowledge of the sandpile PDE.
	
	A key step in the proof is the application of certain discrete derivative bounds  --- Corollary \ref{cor:derivative_bound} and Lemma \ref{topple_limit} below. The strength of these bounds deteriorates as the initial condition grows and it is only when $s_0 \equiv 2d$ that they are strong enough to ensure dimensional reduction. 
	
	This dependence on the initial condition provides some explanation for what occurs in the single-source version. Fey-Levine-Peres conjectured 
	that when 	the initial condition is $(2 d- 2 ) + n \delta_0$  dimensional reduction occurs after carving out a region near the origin. This is likely because 
	our derivative bound (which, in the single-source version, has a flipped dependence on the constant background) does not hold near the origin, but only holds away from it.  Proving this remains open, but we believe the methods in this paper extend to 
	this and other initial data. 
	
	In fact, the result in this paper appears to be a special case of a more general one which we cannot yet prove: dimensional reduction occurs for any uniform initial condition 
	in high enough dimensions. Specifically, for integer $k \geq 0$, when $s_0 \equiv 2d + k$, above a critical dimension, $d > d_0 := k+1$, exact dimensional reduction, modulo $(d_0-1)$-dimensional defects, appears to persist throughout the parallel toppling process. We show that when $d = d_0$, dimensional reduction fails to occur, providing one explanation for why $(d_0-1)$-dimensional defects appear --- see Table \ref{tab:strong_dim1}.

	\subsection{Main Result}
	Our proof begins with the {\it odometer} $v_t$, which encodes the number of topples at each site over time. Let $v_0 : \mathcal{C}_N^{(d)} \to \mathbb{Z}^+$ be the initial odometer $v_0 \equiv 0$; then, recursively, 
	\begin{align}
	v_{t+1} &= v_{t} +  1\{ s_{t} \geq 2 d \}, \\
	s_{t+1} &= s_{t} + \Delta ( v_{t+1} - v_{t} ) ,
	\end{align}
	where $\Delta$ is the graph Laplacian on $\mathcal{C}_N^{(d)}$ with dissipating boundary conditions.  Dependence on $d$ and $N$ is indicated by writing $v_t^{(d,N)}$ and $s_t^{(d,N)}$. 

	We prove dimensional reduction and self-similarity of sandpiles
	via an analysis of the parallel toppling odometers.  It will be more convenient to state the result after making a symmetry reduction. 
	 Let $\operatorname{Aut}_{\mathcal{C}_d}$ denote the group of symmetries of the $d$-dimensional hypercube and let it act on $\Z^d$ by matrix-vector multiplication.
	 The definitions imply that parallel toppling preserves the symmetries of the hypercube: 	$v_t^{(d,N)}(\mathbf{x}) = v_t^{(d,N)}( \sigma \mathbf{x})$ for all $t \geq 1$, $\mathbf{x} \in \mathcal{C}_N^{(d)}$, and $ \sigma \in \operatorname{Aut}_{\mathcal{C}_d}$ (we give a simple proof in Section \ref{subsec:symmetry} below).
	 Hence the odometer and sandpile are fully determined by their restrictions to
	a fundamental domain of the hypercube --- we choose a particular one in the statement of the theorem.

	\begin{theorem}\label{the_theorem}
		Let $N \geq 1$, $d \geq 2$, and $M = \lceil N/2 \rceil$. Denote the fundamental domain of $\mathcal{C}_N^{(d)}$
		consisting of sorted coordinates in decreasing order by $\mathcal{S}^{(d)}_M := \{ (x_1, \ldots, x_d): M \geq x_1 \geq \cdots \geq x_d \geq 1 \}$.
		
		\begin{enumerate}
			\item{Dimensional reduction:} for all $t \geq 1$ and $\mathbf{x}_{d-1} \in \mathcal{S}_M^{(d-1)}$
			\begin{equation*}
			v_t^{(d, N)}(\mathbf{x}_{d-1}, 1) =  v_t^{(d-1, N)}(\mathbf{x}_{d-1})
			\end{equation*}
			and for $\mathbf{x}_{d-1} \geq 2$ 
			\begin{equation*}
			v_{\infty}^{(d,N)}(\mathbf{x}_{d-1}, 2) =  v_{\infty}^{(d-1,N)}(\mathbf{x}_{d-1}).
			\end{equation*}
			\item{Self-similarity:} for all $j < M$, $t \leq \tau_j$, and $\mathbf{x} \in \mathcal{S}_M^{(d)}$ with $\mathbf{x} > M-j$
			\begin{equation*}
			v_t^{(d, N)}(\mathbf{x}) = v_t^{(d, 2 j)}(\mathbf{x}-(M-j)),
			\end{equation*}
			where 
			\begin{equation*}
			\tau_{j} := \min\{ t \geq 1: v_t^{(d,2 j)}(\mathbf{x}) \geq  j  \mbox{ for $\mathbf{x} \in \partial \mathcal{C}^{(d)}_{2 j}$} \},
			\end{equation*}
			and $\partial \mathcal{C}^{(d)}_i := \{ \mathbf{x} \in \mathcal{C}^{(d)}_i : \mbox{there is $\mathbf{y} \not \in \mathcal{C}^{(d)}_i$ with $|\mathbf{y}-\mathbf{x|} = 1$}\}$ denotes the inner boundary of the cube.
			
		\end{enumerate}

	\end{theorem} 

	Both of these results immediately translate to the sandpile. 
	\begin{corollary} 
	
		The results in Theorem \ref{the_theorem}, using the same notation, extend to the sequence of sandpiles. 	
		\begin{enumerate}
			\item{Dimensional reduction:} for all $\mathbf{x}_{d-1} \in \mathcal{S}_M^{(d-1)}$ with $\mathbf{x}_{d-1} \geq 2$
			\[
			s_{\infty}^{(d,N)}(\mathbf{x}_{d-1},1)  = s_{\infty}^{(d-1,N)}(\mathbf{x}_{d-1})+2.
			\]
			\item{Self-similarity:} for all $j < M$, $t \leq \tau_j$, and $\mathbf{x} \in \mathcal{S}_M^{(d)}$ with $\mathbf{x} > M-j+1$
			\[
			s_t^{(d,N)}(\mathbf{x}) = s_t^{(d,2 j)}(\mathbf{x}-(M-j)).
			\]
			\end{enumerate}

	\end{corollary}
	
	The first part of Theorem \ref{the_theorem} states that the parallel toppling odometer restricted to a center slice of a $d$-dimensional hypercube coincides 
	with the $(d-1)$-dimensional parallel toppling odometer for all time.  Dimensional reduction holds for the final odometer off the center, implying dimensional reduction 
	on the center for the final, stable sandpile. 
	
	The second part of the theorem relates the parallel toppling odometers for different sized cubes: 
	the odometer for the size $N$ cube contains the size $2 j$ cube odometer up until the first time a boundary site exceeds $j$ topples. These automatically imply, after shrinking the domains by one, the same for the sequence of sandpiles.

	\begin{table}      
		\begin{tabular}{cccccc}
			\toprule
			$k$ & 0  & 1 & 2 & 3 & 4 \\
			\midrule
			$s_{\infty}^{(2)}(\cdot, \cdot)$ & \includegraphics{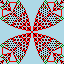}  & \includegraphics{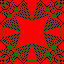} & \includegraphics{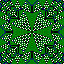}   & \includegraphics{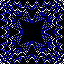}  &  \includegraphics{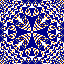} \\
			$s_{\infty}^{(3)}(\cdot, \cdot,1)$ & \includegraphics{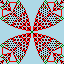}  & \includegraphics{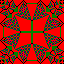} & \includegraphics{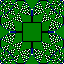}   & \includegraphics{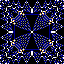}  &  \includegraphics{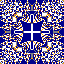} \\
			$s_{\infty}^{(4)}(\cdot, \cdot,1,1)$ & \includegraphics{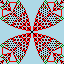}  & \includegraphics{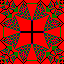} & \includegraphics{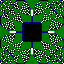}   & \includegraphics{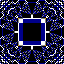}  &  \includegraphics{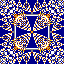} \\
			$s_{\infty}^{(5)}(\cdot, \cdot,1,1,1)$ & \includegraphics{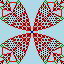}  & \includegraphics{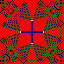} & \includegraphics{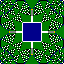}   & \includegraphics{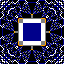}   &  \includegraphics{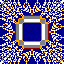} \\
			$s_{\infty}^{(6)}(\cdot, \cdot,1,1,1,1)$ & \includegraphics{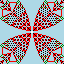}  & \includegraphics{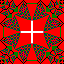} & \includegraphics{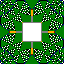}   & \includegraphics{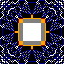}   &  \includegraphics{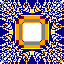} \\
			\bottomrule
		\end{tabular}
		\caption{Center slices of terminal sandpile configurations for $s_0 \equiv 2 d + k$ and $N = 64$. Site colors are normalized by column so that in dimension $d$ a site with $z$ chips has the same color as a site with $(z-2)$ chips in dimension $(d-1)$.}
		\label{tab:strong_dim1}
	\end{table}

	As mentioned previously, we expect Theorem \ref{the_theorem} to be a special case of a more general result. 
	In Section \ref{sec:base_case} we show that dimensional reduction does not occur when $s_0 \equiv 2 d + (d-1)$ 
	 in all dimensions $d \geq 2$ when $N=2$. We also provide an explicit description of the parallel toppling odometer 
	when $s_0 \equiv 2 d + (d-1)$ and $N=4$ for all $d \geq 1$. This explicit form suggests that the proof template in this paper 
	may help with the following. 
	
	\begin{problem} \label{prob:true_theorem}
	 Show that Theorem \ref{the_theorem} holds when $s_0 \equiv 2 d + k$
	 for all $k \geq 0$ and $d > d_0 := (k+1)$. 
	\end{problem}
	
	We expect an even stronger result to be true, although it is likely the proof will require techniques beyond
	those presented here. In simulations, exact dimensional reduction appears to occur away from the central slice. For example, when $s_0 \equiv 4$, $d=2$, and $N$ is large, the center of the sandpile
	contains large curved triangles of 3s. In fact, for every dimension and size we could simulate, 
	whenever dimensional reduction occurs along the central cross sections of the hypercube, it extends; see Figure \ref{fig:strong_reduction} for an example in three dimensions.
	
		\begin{figure}[tbp]
		\centering
		\fbox{\includegraphics[width = 0.44 \textwidth]{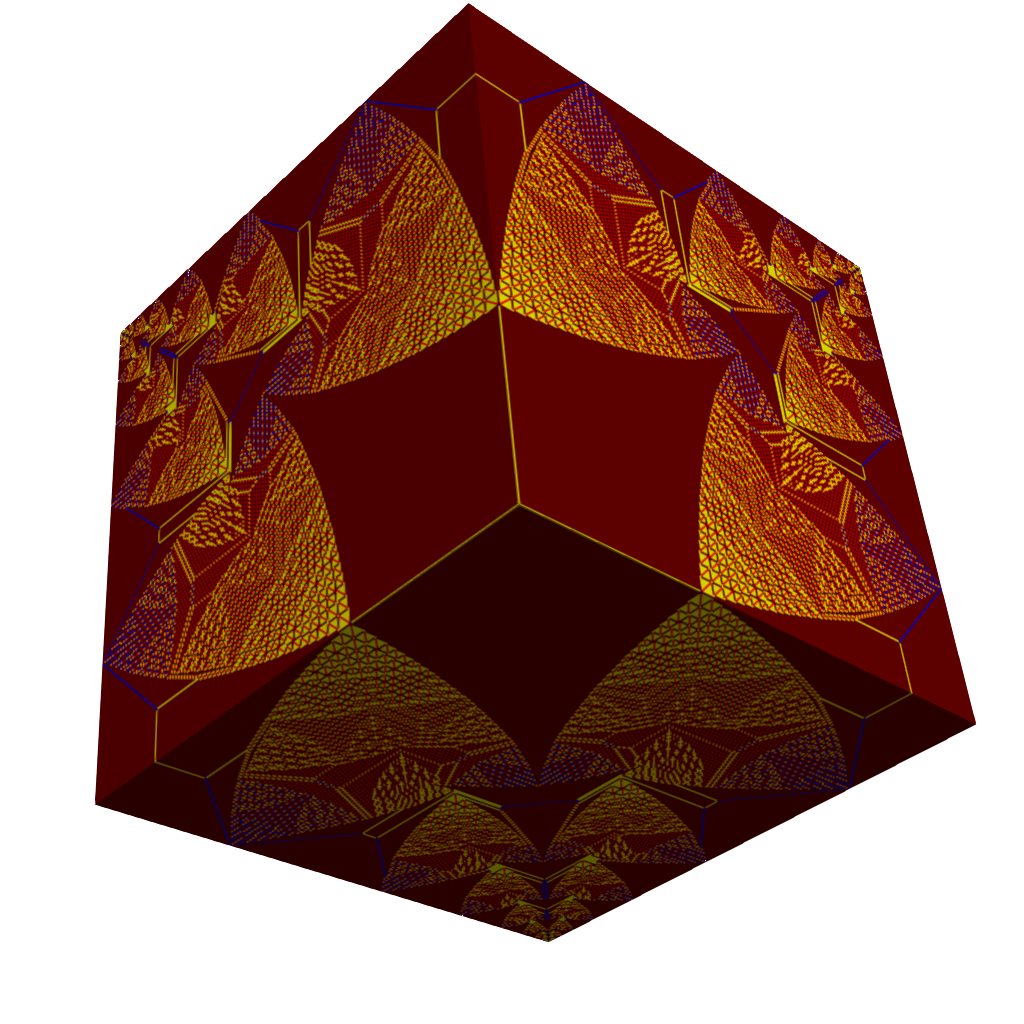}}
		\fbox{\includegraphics[width = 0.44 \textwidth]{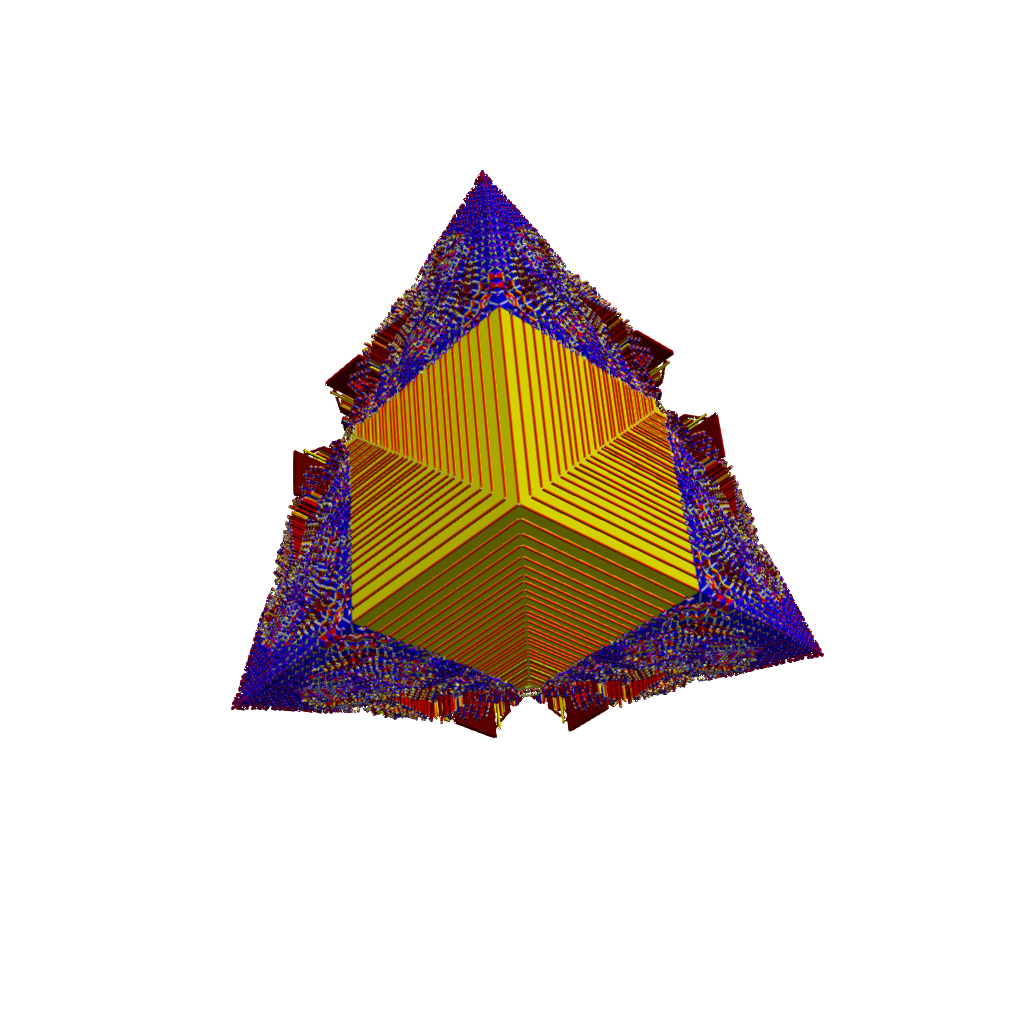}}
		\caption{On the left, a corner cube of $s_{\infty}^{(N,3)}$ and on the right the same corner with all cross sections which match 
			$s_{\infty}^{(N,2)}$ removed. The initial condition is $s_0 \equiv 7$ and $N=2^{11}$. Sites with $0,\ldots, 5$ chips are represented by different colors.}
		\label{fig:strong_reduction}
	\end{figure}

	\begin{problem}
		Extend dimensional reduction on the hypercube to a domain of codimension zero. 
	\end{problem}

	\noindent The following is closely related. 
	\begin{problem} \label{prob:c11}
		Show that the odometer for any bounded initial sandpile on the hypercube has bounded second differences. For instance, show 
		for all $t \geq 1$ and $i=1,2$ that 
		\[
		-3 \leq -2 v_t(\mathbf{x}) + v_t(\mathbf{x} + e_i) +v_t(\mathbf{x} - e_i) \leq 2
		\]
		when $d = 2$ and $s_0 \equiv 4$. 
	\end{problem}	
	Numerical evidence indicates the hypercube is a necessary hypothesis in Problem \ref{prob:c11}. In fact, in most other domains, including the discrete circle, the odometer does not appear to have bounded second differences. 
	On the hypercube, our proof of dimensional reduction shows that the odometer has bounded second differences along the central cross sections; however, a method to propagate those bounds to the interior remains out of reach.

	\subsection{Outline of the proof}
	The proof of Theorem \ref{the_theorem} is a careful induction on hypercube dimension, side length, and time. Some parts of the argument can be simplified 
	but we present it in this fashion to suggest a template for proving dimensional reduction with more general initial data. At a high level, 
	we show that if the parallel toppling process for a $d_0$-dimensional sandpile is sufficiently regular, then dimensional reduction is guaranteed 
	in all dimensions $d > d_0$. We prove this regularity when $d_0=1$; the case $d_0 > 1$ remains open.

	Our main technical tool is a technique introduced by Babai and Gorodezy to prove discrete quasiconcavity of the single-source sandpile odometer in $\Z^2$ \cite{babai2007sandpile}. By an iteration of their technique, we gain symmetry of the odometer, a derivative comparison result, and a parabolic least action principle. These results, which appear in Section \ref{sec:prelims}, extend beyond the hypercube and so may be of independent interest.

	In Section \ref{sec:base_case} we explicitly determine $v_t^{(d,N)}$ when $N=4$ in all dimensions $d \geq 1$ when the initial sandpile is $s_0 \equiv 2 d + d-1$. This is done by 
	mapping the hypercube to a line via a radial decomposition. The explicit form of $v_t$ provides both a base case for our proof and progress towards Problem \ref{prob:true_theorem}.	We also show that when $N=2$ dimensional reduction does not occur at the critical dimension $d_0$. 
	
	The explicit solution when $N=4$ establishes the base case for an odometer regularity result which is then proved in Section \ref{sec:base_regularity}. Finally, in Section \ref{sec:dim_reduction}, we use the established regularity of the odometer in dimension $(d-1)$ to prove dimensional reduction in dimension $d$ 
	
		\subsection*{An efficient algorithm for computing high-dimensional sandpiles}
	In Section \ref{subsec:symmetry} we show that $v_t^{(d,N)}$ can be computed via the parallel toppling procedure restricted to the simplex. In fact, the argument shows that any sandpile with a symmetric initial condition on $\Z^d$, including the single-source sandpile, retains symmetry throughout the parallel toppling process and can be computed in this way. 
	
	For $d$ large, computing sandpiles on the simplex improves space complexity by a factor of $d^d$. Moreover, the reduction in size also leads to a faster algorithm when using parallelization. We wrote a program in Julia \cite{bezanson2017julia, besard2018effective} for computing arbitrary dimensional sandpiles which implements these improvements. The program, which may be freely used and modified, is included in the arXiv post.

	\subsection*{Acknowledgements}	
	It is a pleasure to acknowledge Charles K. Smart for encouraging me to work on this problem and for many valuable conversations throughout.
	Additionally, I thank Lionel Levine for helpful discussions and for sharing past, unpublished, joint work with Alexander Holroyd and Karl Mahlburg towards the dimensional reduction conjecture. The anonymous referees provided detailed feedback which led to a much improved exposition.

	\section{Preliminaries} \label{sec:prelims}

	\subsection{Notation and conventions}
	When we need to distinguish between vectors and scalars, we reserve $\mathbf{x}$
	for vectors and $x$ for scalars. The $i$-th element of $\mathbf{x}$ is $x_i$
	and $\mathbf{x}_{i} = (x_1, \ldots, x_i)$. We refer to coordinate basis vectors of 
	$\Z^d$ as $e_i$ and the ones vector of length $i$ by
	$\mathbf{1}_{i} = (1, \ldots, 1)$. Equalities, inequalities, addition, and multiplication between vectors and scalars are to be understood pointwise. 
	The notation $|x|$ refers to the absolute value and $|\mathbf{x}| = \sum_{i=1}^d |x_i|$ is the $\ell_1$ norm.

	We embed $\mathcal{C}_{N}^{(d)}$ into $\Z^d$ in two different ways depending on whether $N$ is even or odd. If $N = 2 k + 1$,  $\mathcal{C}_N^{(d)} = \{ \mathbf{x} +1 : |x_i| \leq k \}$  otherwise $\mathcal{C}_N^{(d)} = \{ \mathbf{x}-k : 1 \leq x_i \leq 2 k \}$.
	The graph Laplacian on $\mathcal{C}_N^{(d)}$ operates on functions
	$f: \mathcal{C}_N^{(d)} \to \mathbb{R}$ as 
	\begin{equation}
	\Delta^{(d)} f(x) = -2 d f(x) + \sum_{y \sim x} f(y), 
	\end{equation}
	where we pad $f(x) := 0$ for $x \not \in \mathcal{C}_N^{(d)}$ and the 
	sum $y \sim x$ is over the 2d nearest neighbors of $x$, $|y-x| = 1$.  When the hypercube size or dimension is not used, we omit distinguishing sub/superscripts.

	\subsection{Babai-Gorodezky technique} 
	In this subsection and the next, let $s_0 \leq 2 (2 d) - 1$ be an arbitrary initial sandpile on $\mathcal{C}$ and $v_t$ its odometer. A straightforward induction argument and the definition of the graph Laplacian yields the following lemma. 
	\begin{lemma}[] \label{pt_ind_lemma}
		For each $x \in \mathcal{C}$ and all $t \geq 0$, 
		\begin{equation}
		v_{t+1}(x) =   \lfloor \frac{ s_0(x) + \sum_{y \sim x} v_{t}(y) }{2 d} \rfloor.
		\end{equation}
	\end{lemma}
	Babai and Gorodezky used this simple lemma to prove a nontrivial discrete quasiconcavity property of the single-source sandpile in $\Z^d$ \cite{babai2007sandpile}. A more general version of their argument appears below in Lemma \ref{general_difference_comparison}. Roughly, their technique recognizes that if a property of the odometer holds at $t=1$, is consistent across the symmetry axes, and can be verified on the boundaries of the domain, it must hold for all $t \geq 1$.
	
	Lemma \ref{pt_ind_lemma} is used many other times throughout this paper; notably we use it to prove a parabolic least action principle and symmetry of the odometer on the hypercube.

	\subsection{Parabolic least action principle} 
	The least action principle \cite{fey2010growth} shows that $v_\infty$ is minimal among all $w: \mathcal{C} \to \mathbb{Z}^+$ which {\it stabilize} $s_0$: $\Delta w + s_0 \leq 2d -1$. We upgrade this to a parabolic least action principle by observing the parallel toppling procedure as a directed sandpile on $\mathcal{G} = \mathbb{Z}^+ \times \mathcal{C}$. The initial sandpile and odometer over time are stacked,
	$s(t,x) := s_0(x)$ and $v(t,x) := v_t(x)$ for all $t \geq 0$ and $x \in \mathcal{C}$.  The graph Laplacian operates on functions $f:\mathcal{G} \to \mathbb{R}$ as
	\begin{equation}
	\Delta^{\mathcal{G}} f(t,x) = -2 d f(t,x) + \sum_{y \sim x} f(t-1,y),
	\end{equation}
	for $t \geq 1$ and $x \in \mathcal{C}$, where the sum $y \sim x$ is over the nearest neighbors of $x$ in $\mathcal{C}$.

	\begin{lemma}[Parabolic least action principle] \label{parabolic_comparison}
		\begin{equation} \label{eq:parabolic_leastaction}
		v(t,x) = \min\{ u: \mathcal{G} \to \mathbb{Z}^+: \Delta^{\mathcal{G}} u(t,x) + s(t,x) \leq 2d-1 \mbox{ for all $x \in \mathcal{C}$ and $t \geq 1$} \}
		\end{equation}
	\end{lemma}
	\begin{proof}
		Let $w(t,x)$ denote the right-hand side of \eqref{eq:parabolic_leastaction}. We show using Lemma \ref{pt_ind_lemma} and induction that $v(t,x) = w(t,x)$. Indeed, by the directed structure of $\mathcal{G}$, it suffices to show this equality one time slice at a time. Equality holds at $t=0$ as $v(0,x) = w(0,x) = 0$. Assume that $v(t',\cdot) = w(t', \cdot)$ for $t' \leq t$ and let $x \in \mathcal{C}$ be given. 
		The monotonicity of the graph Laplacian implies $\Delta^{\mathcal{G}} w + s \leq 2 d- 1$, hence, 
		\[
		\Delta^{\mathcal{G}} w(t+1,x) + s(t+1,x) = -2d w(t+1,x) + s_0(x) + \sum_{y\sim x} w(t,y) < 2d, 
		\]
		and a rearrangement yields, 
		\[
		w(t+1,x) \geq \lfloor \frac{s_0(x) + \sum_{y\sim x} w(t,y)}{2d} \rfloor. 
		\]
		By Lemma \ref{pt_ind_lemma} and the inductive hypothesis, the right-hand side of the above is exactly $v(t+1,x)$. Similarly, for the other direction, 
		\begin{align*}
		\Delta^{\mathcal{G}} v(t+1,x) + s(t+1,x) &=  2 d \left( \frac{s_0(x) + \sum_{y\sim x} w(t,y)}{2d} - \lfloor \frac{s_0(x) + \sum_{y\sim x} w(t,y)}{2d} \rfloor \right) \\
		&< 2d,
		\end{align*}
		which concludes the proof by minimality of $w$. 
	\end{proof}
	\noindent  
	Our usage of the parabolic least action principle in the main argument is minimal
	and can be avoided. And, in some sense, it is a restatement of Lemma \ref{pt_ind_lemma}. We included it as it may be of independent interest. 
	
	\subsection{Symmetry and fundamental domains} \label{subsec:symmetry}
	In this section we observe that sandpile dynamics on $\mathcal{C}_N^{(d)}$ preserve the symmetry structure of the $d$-dimensional hypercube.
	This is then used to reduce to the sandpile on a {\it fundamental domain} of the hypercube with reflecting boundary conditions.  The main contribution 
	of this subsection is a coordinate-wise description of this domain along with an explicit formula for the reflecting graph Laplacian. 
	
	We briefly provide a presentation of the group of automorphisms of the hypercube
	and its action on $\Z^d$; for more details see, for example, \cite{godsil2013algebraic}. 
	Let $\operatorname{Aut}_{\mathcal{C}_d}$ be the group of $(d \times d)$ matrices 
	with exactly one $\pm 1$ in each row and in each column and $0$s elsewhere. 
	Let $\sigma \in \operatorname{Aut}_{\mathcal{C}_d}$ act on $x \in \Z^d$ by matrix-vector multiplication
	followed by a translation and let it act on $f:\Z^d \to \mathbb{R}$ by $\sigma f(x) := f(\sigma x)$. 
	The translation is chosen to preserve $\mathcal{C}_{N}^{(d)}$ when $N$ is even or odd in our choice of coordinates. 

	
	Each $\sigma$ is an isometry and hence preserves nearest neighbors and $\mathcal{C}^{(d)}_N$. 
	That is, if $y \not \in \mathcal{C}^{(d)}_{N}$, then $\sigma y \not \in \mathcal{C}^{(d)}_{N}$.
	And, if $|y-x| = 1$, then $|\sigma y - \sigma x| = 1$, so 
	\begin{equation} \label{eq:isometry}
	\sum_{y' \sim \sigma x} f(y') = \sum_{y \sim x} f(\sigma y).
	\end{equation}

	We say $\Omega \subseteq \mathcal{C}_{N}^{(d)}$ is a fundamental domain if there exists a set $\{\sigma_1, \sigma_2, \ldots \}  \subset  \operatorname{Aut}_{\mathcal{C}_d}$ so that $\bigcup_j \sigma_j \Omega = \mathcal{C}_{N}^{(d)}$. For example, a fundamental domain of an interval is half of it, 
	while a fundamental domain of a square is a right triangle with one side along a central cross section of the square. The fundamental domains which we consider have coordinate consistency across dimensions.  Let $M = \lceil N/2 \rceil$ and 
	\begin{equation}
	\mathcal{S}^{(d)}_M := \{ (x_1, \ldots, x_d): M \geq x_1 \geq \cdots \geq x_d \geq 1 \}.
	\end{equation}
	Observe that  $\{ \mathbf{x}_{d-1} : (\mathbf{x}_{d-1}, 1) \in \mathcal{S}^{(d)}_M\} = \mathcal{S}^{(d-1)}_M$;
	this is the first step towards proving dimensional reduction on the hypercube. 
	
	Let $v_t$ be the odometer function for an initial sandpile, $s_0$, on $\mathcal{C}_{N}^{(d)}$ which is {\it symmetric},  $ \sigma s_0 = s_0$ for all $ \sigma \in \operatorname{Aut}_{\mathcal{C}_d}$. We show that the parallel toppling odometer coincides 
	with the {\it symmetrized} odometer $v_t^{\mathcal{S}}$ on $\mathcal{S}_M^{(d)}$ with 
	appropriate reflecting boundary conditions. That is, for each $x \in \mathcal{S}_M^{(d)}$ and $y  \sim x$ there exists a unique rotation or reflection, $ \sigma_y \in \operatorname{Aut}_{\mathcal{C}_d}$, so that $ \sigma_y y \in \mathcal{S}_M^{(d)}$.
	Let $y \overset{\mathcal{S}}{\sim} x$ denote iteration over the set $\{  \sigma_y  y : y \sim x\}$.  For all $t \geq 0$ and $x \in \mathcal{S}_M^{(d)}$, let 
	\begin{equation}
	v_{t+1}^{\mathcal{S}}(x) = \lfloor \frac{ s_0(x) +  \sum_{y \overset{\mathcal{S}}{\sim}  x} v_t^{\mathcal{S}}(y)}{2 d} \rfloor, 
	\end{equation}
	where $v_0^{\mathcal{S}}: \equiv 0$. 
	
	We provide an algorithmic construction of this which we use to prove Lemma \ref{general_difference_comparison} below. Suppose $M \geq 2$, let $x \in \mathcal{S}_M^{(d)}$ be given and define 		
	$x_0 = M$ and $x_{d+1} = 1$.  The following algorithm produces a sequence of indices describing the symmetrized nearest neighbors of $x$. 
	Start with $l_0 = 0$ and pick the largest $(d+1) \geq r_0 \geq l_0$ with 
	$x_{l_0} = x_{r_0}$. If $r_0 = (d+1)$, stop, otherwise, set $l_1 = (r_0 + 1)$
	and repeat, generating
	\begin{equation}
	\mathbf{I}^{(M,d)}(x) := \{ (l_0, r_0), \ldots, (l_n, r_n) \}, 
	\end{equation}
	where $n \leq (d + 1)$. Observe that 
\[	
M = x_{l_0} = x_{r_0} <   \cdots < x_{l_k} = x_{r_k} < \cdots < x_{l_n} = x_{r_n} = 1
\]
	so that  
	\begin{equation}
	\begin{aligned}
	\sum_{y \overset{\mathcal{S}}{\sim}  x} v_t(y) &= \sum_{k=1}^{n-1} (1 + r_k-l_k)\left( v_t(x - e_{r_{k}}) + v_t(x + e_{l_{k}}) \right) \\
	&+ (r_0 - l_0)  v_t(x - e_{r_0}) \\
	&+ (r_n - l_n)  ( v_t(x - e_{r_n}) + v_t(x + e_{l_n}))
	\end{aligned}
	\end{equation}
	where $v_t(x - e_0) = v_t(x + e_{d+1}) := 0$ and 
	\begin{equation}
	v_t(x - e_{r_n}) =  \begin{cases}
	v_t(x) & \mbox{ if $N$ is even } \\
	v_t(x+e_{l_{n}}) & \mbox{ if $N$ is odd}.
	\end{cases}
	\end{equation}


	\begin{lemma}[Symmetry]\label{symmetry_lemma}
		For each $t \geq 0$ and each $ \sigma \in \operatorname{Aut}_{\mathcal{C}_d}$, $ \sigma v_t = v_t$. Hence, $v_t = v_t^{\mathcal{S}}$ on $\mathcal{S}^{(d)}_M$.
	\end{lemma}
	
	\begin{proof}
		We prove symmetry of $v_t$ by induction and Lemma \ref{pt_ind_lemma}. 
		At $t = 0$, $v_t \equiv 0$, so suppose symmetry holds at time $t$. Let $ \sigma \in \operatorname{Aut}_{\mathcal{C}_d}$, $x \in \mathcal{C}_{N}^{(d)}$ be given. By Lemma \ref{pt_ind_lemma},
		 \eqref{eq:isometry}, and the inductive hypothesis, 
		\begin{align*}
		v_{t+1}( \sigma x) &= \lfloor \frac{ s_0( \sigma x) + \sum_{y' \sim  \sigma x} v_t( y')}{2d} \rfloor \\
		&= \lfloor \frac{ s_0( \sigma x) + \sum_{y \sim x} v_t( \sigma y)}{2d} \rfloor \\
		&= \lfloor \frac{  s_0(x) + \sum_{y \sim x} v_t(y)}{2d} \rfloor \\
		&= v_{t+1}(x). 
		\end{align*}
	\end{proof}
	\noindent Note that the proof indicates that Lemma \ref{symmetry_lemma} can be extended in a natural way to 
	other graphs and domains which are preserved under the automorphism group of the graph.

	Henceforth, we consider $v_t^{\mathcal{S}}$ in $\mathcal{S}_M^{(d)}$ and drop all $\mathcal{S}$ superscripts. To reduce the number of cases with similar 
	arguments, we only consider $N = 2 M$. Indeed, when $N$ is odd, the proofs
	are identical except for slight changes to the boundary arguments. Also, we will use $\mathcal{S}_M^{(d)}$ to refer exclusively to the sorted fundamental domain of $\mathcal{C}_{2M}^{(d)}$. The expressions $v_t^{(d,M)}$  and $s_t^{(d,M)}$ will refer to the parallel toppling odometers and sandpiles on $\mathcal{S}_M^{(d)}$.

	\subsection{Derivative comparison}
	In this section we provide a general parabolic comparison result for first order differences of $v_t$ on $\mathcal{S}_M$
	when the initial sandpile, $s_0$, is constant. For $w \in \Z^d$, let $D_{w}$ denote a first order difference operator of the form $D_{w} v_t(\cdot) = v_t(\cdot) - v_t(\cdot+w)$. Pad $v_t$ by $v_t(x) := 0$ for all $x \not \in \mathcal{C}_N$.  
	Denote the interior with respect to $w$ as 
	\begin{equation}
	\mbox{Int}_w(\mathcal{S}_M) = \{ x \in \mathcal{S}_M : y' \in \mathcal{S}_M \mbox{ for all $|y'-(x+w)| = 1$} \}
	\end{equation}
	and the boundary as 
	\begin{equation}
	\partial_w \mathcal{S}_M = \mathcal{S}_M \backslash \mbox{Int}_w(\mathcal{S}_M).  
	\end{equation}

	Observe that every symmetrized $y \sim x \in \mathcal{S}_M$ is of the form $y = (x \pm e_i) \mp d_i$,  where $d_i$ is either a reflection, $d_i = e_i$ or a 
	rotation $d_i = e_i -e_j$. We will show that if one can control $D_{w} v_t$ over the reflecting, rotating, and dissipating boundaries of $\mathcal{S}_M$, then that control persists over time. The dissipating boundary on $\mathcal{S}_M$ is
	\begin{equation} \label{eq:dissipative_boundary}
	\mathcal{B}^{(disp)}_w \mathcal{S}_M = \{ x \in \mathcal{S}_M : (x + w)_{i} \geq M  \mbox{ for some $1 \leq i \leq d$} \}
	\end{equation}
	while the reflecting and rotating boundaries are  
	\begin{equation} \label{eq:reflecting_boundary}
	\mathcal{B}^{(ref)}_w \mathcal{S}_M = \{ x \in \mathcal{S}_M : (x+w)_{i} \leq 1 \mbox{ for some $1 \leq i \leq d$} \}
	\end{equation}
	and
	\begin{equation} \label{eq:rotating_boundary}
	\mathcal{B}^{(rot)}_w \mathcal{S}_M = \{ x \in \mathcal{S}_M : (x+w)_i \leq (x + w)_j \mbox{ for some $1 \leq i < j \leq d$} \}.
	\end{equation}
	For notational convenience write
	\begin{equation} \label{eq:all_boundaries}
	\mathcal{B}_w = \{\mathcal{B}^{(disp)}_{w}  \cup \mathcal{B}^{(ref)}_{w} \cup \mathcal{B}^{(rot)}_{w} \}
	\end{equation}
	and  $\mathcal{B}_{w_1, \ldots, w_n} = \cup_{i=1}^n \mathcal{B}_{w_i}$
	for points $w_i \in \Z^d$. 
		Note that in next lemma, we employ our convention to sometimes omit distinguishing sub/superscripts. 
	
	\begin{lemma} \label{general_difference_comparison}
		Let  $\mathbf{w} := \{w_1, \ldots, w_n\}$ be a set of points in $\Z^d$ each equipped with a function $g_j: \mathcal{S} \to \Z$ which is superharmonic in the interior of $\mathcal{S}$. If	
		\begin{equation}
		\label{eq:initial_cond} \sup_{j} \left( D_{w_j} v_{t_0}(x) - g_j(x) \right)  \leq 0 \mbox{\qquad for all $x \in \mathcal{S}$}
		\end{equation}
		and for all $t \geq t_0$ and $x \in \{\mathcal{B}_0 \cup \mathcal{B}_{\mathbf{w}}\} \mathcal{S}$,
		\begin{equation} \label{eq:boundaries_ref}
		\begin{aligned}
		\sup_{j} &( \sum_{y \sim x} v_{t}(y) - \sum_{y' \sim (x + w_j)} v_{t}(y') -  2 d g_j(x) ) \leq 0 \\
		&\mbox{  \qquad or } \\
		\sup_{j}  &( D_{w_j} v_{t+1}(x) - g_j(x) ) \leq 0
		 \end{aligned}
		\end{equation}
		then
		\begin{equation} \label{eq:derivative_comparison}
		\sup_{j} \left( D_{w_j} v_{t+1}(x) - g_j(x) \right) \leq 0
		\end{equation}
		for all $t \geq t_0$ and $x \in \mathcal{S}$. 
		
	\end{lemma}
	
	\begin{proof}
		We prove this by induction on $t$, starting at $t_0$, the base case guaranteed by \eqref{eq:initial_cond}.
		Suppose $\eqref{eq:derivative_comparison}$ holds at $t$ and let $w_j$, $x \in \mathcal{S}$ be given. 
		First suppose $x \in \{ \mbox{Int}_{w_j} \cap \mbox{Int}_{0} \}(\mathcal{S})$. By Lemma \ref{pt_ind_lemma}
		\begin{align*}
		D_{w_j} v_{t+1}(x) - g_j(x) &\leq \lfloor \frac{(2d -1) + \sum_{y \sim x} v_t(y) - \sum_{y' \sim (x + w_j)} v_t(y')}{2d} \rfloor -  g_j(x) \\
		&=  \lfloor \frac{(2d -1) + \sum_{y \sim x} D_{w_j} v_{t}(y) }{2d} \rfloor -  g_j(x)  \\
		&=  \lfloor \frac{(2d -1) + \sum_{y \sim x} \left( D_{w_j} v_{t}(y) - g_j(y) \right) + \sum_{y \sim x} (g_j(y)-g_j(x))  }{2d} \rfloor  \\
		&\leq  0 
		\end{align*} 
		as $g_j$ is superharmonic and integer-valued.  If $x \in \{ \mathcal{B}_0 \cup \mathcal{B}_{w} \}$, then we either use the same argument or conclude depending on the case in \eqref{eq:boundaries_ref}.

	\end{proof}
	
	\noindent As a corollary, we deduce the following discrete quasiconcavity property of $v_t$ on a hypercube, which was proved in \cite{babai2007sandpile}
	for axis monotonic initial sandpiles on $\Z^2$.  (Note that Aleksanyan and Shahgholian, using a discrete analogue of the method of moving planes, proved axis monotonicity
	of $v_{\infty}$ in \cite{aleksanyan2019discrete}.)
	\begin{corollary}[Axis monotonicity \cite{babai2007sandpile}] \label{axis_monotonicity}
		For all $t \geq 1$, $x \in \mathcal{S}$, and all sets of indices 
		\[
		\mathcal{I} = \{ i_1, \ldots, i_n : 1 \leq i_1  < \cdots < i_{n} \leq d \}
		\]
		and
		\[
		\mathcal{J} = \{ j_1, \ldots, j_m : i_n < j_1 < \cdots < j_{m} \leq d \} 
		\]
		where $n \geq 1$ and $m \geq 0$ we have 
		\[
		v_t(x) \geq v_t(x + e_{\mathcal{I}} - e_{\mathcal{J}}) \quad \mbox{for $(x + e_{\mathcal{I}} - e_{\mathcal{J}}) \in \mathcal{S}$}
		\]
		where $e_{\mathcal{I}} = \sum_{i \in \mathcal{I}} e_i$ denotes a sum over standard basis vectors indexed by $\mathcal{I}$.  
	\end{corollary}

	We also have control on the derivative given an odometer upper bound
	on the dissipating boundary. 
		\begin{corollary}[Derivative bound] \label{cor:derivative_bound}
		Suppose $v_{\infty}(M, \mathbf{1}_{d-1}) \leq k M$ for integer $k \geq 1$. 
		Then, for all $1 \leq j \leq d$ and $t \geq 0$
		\begin{equation} \label{eq:cor:derivative_bound}
		v_t(x) - v_t(x+e_j) \leq k x_j.
		\end{equation}

	\end{corollary}
	
	\begin{proof}
	The claim is immediate if $M = 1$, so suppose $M \geq 2$. 
 Let $e_j$ and $x$ be given and let
	\begin{equation}
	\mathbf{I}^{(M,d)}(x) := \{ (l_0, r_0), \ldots,(l_n, r_n)\}, 
	\end{equation}
	be the indices describing the nearest neighbors of $x$ as defined in Section \ref{subsec:symmetry} above. 

	Pick the largest index $J$ so that 
	\[
	l_{J} \leq j \leq r_{J},
	\]	
	$v_t(x + e_j) = v_t(x+e_{l_J})$, and  (recalling $l_J = r_{J-1} + 1$)
	\begin{equation} \label{eq:neighbor_sym}
	\mathbf{I}^{(M,d)}(x + e_{l_J} ) = 
	\begin{cases}
	\{ \ldots,  (l_{J-1}, r_{J-1} + 1), ( l_J + 1, r_J), \ldots \} &\mbox{ if $x_{r_{J-1}} = x_{l_J} + 1$}  \\
	\{ \ldots,  (l_{J-1}, r_{J-1}), ( l_J, l_J),  (l_J+1, r_J), \ldots \} &\mbox{ if $x_{r_{J-1}} > x_{l_J} + 1$} .
	\end{cases}
	\end{equation}

	As $g_j(x) := k x_j$ is harmonic in the interior of $\mathcal{S}$, it remains to check \eqref{eq:boundaries_ref} in Lemma \ref{general_difference_comparison}. 
	For later reference, we label the expression we bound,
	\begin{equation} \label{eq:neighbor_sum}
	\sum_{y \sim x} v_t(y) - \sum_{y' \sim (x+e_{l_J})} v_t(y').  
	\end{equation}
	 The computations are similar in other cases, so we assume $x_{r_{J-1}} = x_{l_J} + 1$ and $l_J + 1 \leq r_J$.

	{\bf Case 1: $J = 0$} \\
	As we are on the dissipating boundary, $v_{t+1}(x+e_{l_J}) = 0$ and $x_{l_J} = \cdots = x_j = M$, hence 
	\[
	v_{t+1}(x) - v_{t+1}(x + e_{l_J})  = v_{t+1}(x) \leq v_{\infty}(M, \mathbf{1}_{d-1}) \leq k M
	\]
	 by axis monotonicity and our assumption on the odometer.

	{\bf Case 2: $1 < J < n$}\\
	We compute \eqref{eq:neighbor_sum}, observing that all differences except for those near $r_J$ are unaffected
	by the symmetrization; 
	\begin{align*}
	\eqref{eq:neighbor_sum} &= \sum_{k=1, k \not \in [J-1,J]}^{n-1} (1 + r_k - l_k) ( v_t(x - e_{r_k}) - v_t(x-e_{r_k} + e_{l_J})) \\
	&+\sum_{k=1, k \not \in [J-1,J]}^{n-1} (1 + r_k - l_k) (v_t(x + e_{l_k}) - v_t(x + e_{l_k} + e_{l_J})) \\
	&+  (r_0 - l_0)  (v_t(x - e_{r_0}) - v_t(x - e_{r_0} + e_{l_J})) \\
	&+ (r_n - l_n)  ( v_t(x)  - v_t(x + e_{l_J})  + v_t(x+e_{l_n}) - v_t(x + e_{l_n} + e_{l_J}))   \\
	&+ \star_{\{J-1,J\}} \\
	&\leq (2 d - 2 (1 + r_J - l_{J-1})) k x_{j}  + \star_{\{J-1,J\}} 
	\end{align*}
	where $\star_{\{J-1,J\}}$ is defined as the sum of terms in the difference with indices $\{J-1,J\}$. 
	This can then be computed,
	\begin{align*}
	\star_{\{J-1,J\}}  &= (1 + r_J - l_J) ( v_t(x - e_{r_J}) + v_t(x + e_{l_J})) \\
	&- (r_J - l_J) ( v_t(x + e_{l_J} - e_{r_J}) + v_t(x +e_{l_J} + e_{l_{J} + 1})) \\
	&+  (1 + r_{J-1} - l_{J-1})( v_t(x - e_{r_{J-1}}) + v_t(x + e_{l_{J-1}})) \\
	&- (2 + r_{J-1} -  l_{J-1}) ( v_t(x + e_{l_J} - e_{l_J}) + v_t(x + e_{l_J} + e_{l_{J-1}})) \\
	&\leq (r_J - l_J) 2 k x_j +  (1 + r_{J-1} - l_{J-1})( k (x_j+1-1) + k x_j) \\
	&+ k (x_{j}-1)  + k(x_j + 1)  \\
	&=  2 (1 + r_J - l_{J-1})  k x_j.
	\end{align*}

	{\bf Case 3: $J = 1 < n$} \\
	We bound differences with indices $\{0,1\}$ in \eqref{eq:neighbor_sum},
	\begin{align*} 
	\star_{\{0,1\}} 
	&\leq  (r_0 -l_0) k x_{l_1} \\
	&+ (r_1 - l_1) 2 k x_{l_1} \\
	&+ ( v_t(x - e_{r_1}) - v_t(x)) +   (v_t(x + e_{l_1})  - 0) \\
	&\leq (r_0 -l_0) k x_{l_1}  +  (r_1 - l_1) 2 k x_{l_1} +k  ( x_{l_1} - 1) + k (x_{l_1} + 1) \\ 
	&\leq    2 (r_1 - l_0)  k x_{l_1}.
	\end{align*}

	{\bf Case 4: $J = n > 1$}  \\
	We bound differences with indices $\{n-1,n\}$ in \eqref{eq:neighbor_sum},
	\begin{align*}
	\star_{\{n-1,n\}} 
	&\leq (1 + r_{n-1} - l_{n-1})( k (x_{r_{n-1}}-  1) + k x_{l_n}) \\
	&+ (r_n  - l_n-1) ( k x_{l_n} + k x_{l_{n} + 1}) \\
	&+ (v_t(x) - v_t(x)) + (v_t(x + e_{l_n}) - v_t(x + e_{l_n}  + e_{l_{n-1}})) \\
	&\leq (1 + r_{n-1} - l_{n-1})2 k x_{l_n} + (r_n  - l_n-1) 2 k x_{l_n} \\
	&+ k (x_{l_{n-1}}) \\
	&\leq  2 (r_n - l_{n-1}) k x_{l_n}.
	\end{align*}
	In the last step we used $x_{l_{n-1}} = x_{l_n} + 1 = 2 x_{l_n}$.

	{\bf Case 5: $J = n = 1$} \\
	We bound differences with indices $\{0,n\}$ in \eqref{eq:neighbor_sum},
	\begin{align*}
	\star_{\{0,n\}} 
	&\leq (r_0 - l_0) (k (x_{r_0} - 1))  \\
	&+ (r_n - l_n - 1) (k x_{l_n} + k x_{l_n+1}) \\
	&+ (v_t(x + e_{l_n}) - 0) + (v_t(x) - v_t(x)) \\
	&\leq (r_0 - l_0) 2 k x_{l_n} + (r_n - l_n - 1) 2 k x_{l_n}   \\
	&+ k x_{l_n} \\ 
	&\leq 2 (r_n - l_0 - 1) k x_{l_n}.
	\end{align*}

	\end{proof}

	\subsection{Weak topple control}	
	We provide a difference-in-time analogue of Lemma \ref{general_difference_comparison}

	\begin{lemma} \label{topple_limit}
	For all $t \geq t_0$ and $j \geq 0$, 
	\[
	\max_{z \in \mathcal{S}} \left( v_{t+j}(z) - v_t(z) \right) \leq \max_{z \in \mathcal{S}} \left( v_{t_0+j}(z) - v_{t_0}(z) \right)
	\]
	\end{lemma}

	\begin{proof}
		We induct on $t$ starting at $t_0$. Suppose the result holds for $(t-1)$
		and let $x \in \mathcal{S}$ be given.
		Lemma \ref{pt_ind_lemma} implies  
		\[
		v_{t + j}(x) - v_{t}(x) \leq \lfloor \frac{ (2d-1) + \sum_{y \sim x} \left(v_{t+j-1}(y) - v_{t-1}(y) \right) }{2 d} \rfloor, 
		\]		
		hence, by induction 
		\[
			v_{t + j}(x) - v_{t}(x) \leq \lfloor \frac{ (2d-1) + 2 d\left( \max_{z \in \mathcal{S}} \left( v_{t_0+j}(z) - v_{t_0}(z) \right) \right) }{2 d} \rfloor =  \max_{z \in \mathcal{S}} \left( v_{t_0+j}(z) - v_{t_0}(z) \right). 
		\]

	\end{proof}

	\section{Explicit solutions when \texorpdfstring{$M \leq 2$}{M <= 2}} \label{sec:base_case}
	
	In this section, we compute $v_t^{(d)}$ when $s_0^{(d)} \equiv 2 d + (d-1)$ for all $d \geq 1$ when $M=2$.
	We also show that dimensional reduction does not occur at dimension
	$d=d_0$ when $s^{(d)}_0 \equiv 2 d + (d_0-1)$ and $M=1$.

	\subsection{\texorpdfstring{$M=1$}{M=1}}
	As we do not know how to define a $0$-dimensional sandpile, suppose $d_0 \geq 2$. 
	\begin{prop}
	 When $s_0^{(d)} \equiv 2 d + (d_0-1)$,  $v_\infty^{(d_0)} \equiv 1$ but $v_\infty^{(d_0-1)} \equiv 2$.
	\end{prop}
	\begin{proof} 
		In dimension $d$, a corner site of the hypercube has $d$ internal neighbors so 
		$\Delta^{(d)} (\mathbf{1}) = - 2 d + d = -d $.
		Hence, in dimension $d_0$, 
		\[
		s_1^{(d_0)}(\mathbf{1}) = (2 d_0 + (d_0-1))- d_0 = 2 d_0 - 1
		\]
		however, in dimension $(d_0-1)$, 
		\[
		s_1^{(d_0-1)}(\mathbf{1}) = (2 (d_0-1) + (d_0-1)) - (d_0-1) = 2(d_0-1).
		\]
	\end{proof}

	\subsection{\texorpdfstring{$M=2$}{M=2}}
	Now, suppose $d \geq 1$. After a radial reparameterization of $\mathcal{S}_2$, arbitrary dimensional sandpiles become one-dimensional with a simple nearest-neighbor 
	toppling rule.  Indeed, every $\mathbf{x} \in \mathcal{S}_2$ is of the form $\mathbf{x} =(\mathbf{2}_{x},\mathbf{1}_{d-x})$, for $x = 0, \ldots, d$. Overload notation and consider $s_t$ and $v_t$ as functions on $\{0, \ldots, d\}$. The Laplacian on the one-dimensional graph can then be computed using symmetry. 
	\begin{lemma}
		If we define $0 = v_t(d+1) = v_t(-1)$, then 
		\[
		\Delta v_t(x) = (-d - x) v_t(x) + (d-x) v_t(x+1) + x v_t(x-1).
		\]
	\end{lemma}
	\begin{proof}
			Let $\mathbf{x} = (\mathbf{2}_{x},\mathbf{1}_{d-x})$ so that $\mathbf{I}^{(2,d)}(\mathbf{x}) = \{ (0,x), (x+1,d+1) \}$.
			Hence, by definition of the symmetric Laplacian, 
\begin{align*}
				\Delta v_t( \mathbf{x}) &= -2 d v_t(\mathbf{x}) + x v_t(\mathbf{x} - e_x) + (d-x) (v_t(\mathbf{x}) + v_t(\mathbf{x} + e_{x+1}))  \\
				&= -2d v_t(x) + (d-x) v_t(x) + (d-x) v_t(x+1) + x v_t(x-1).
\end{align*}

	\end{proof}

	\begin{figure}
		\centering
		\fbox{\includegraphics[width = 0.5\textwidth]{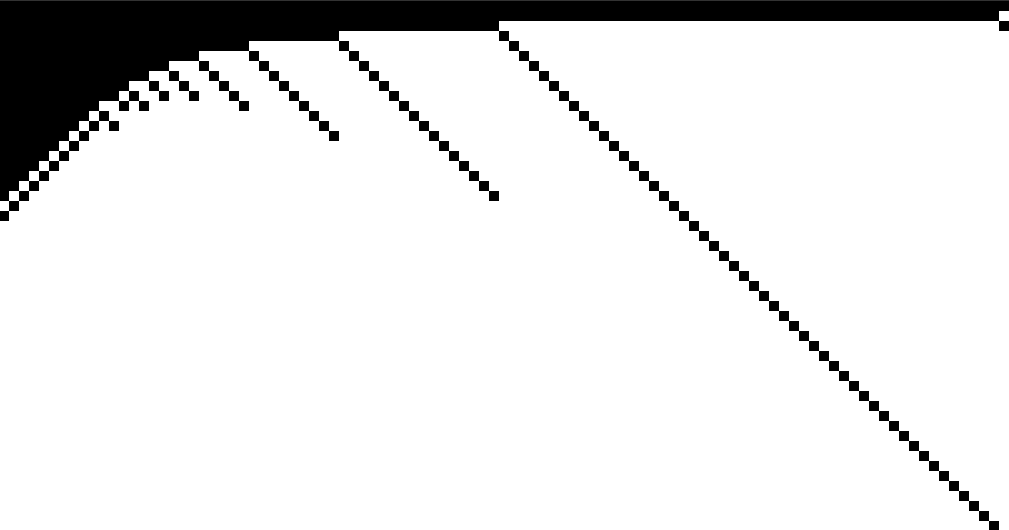}}
		\caption{The parallel toppling odometer for $s_0 \equiv 2 d + (d-1)$ when $d = 100$ and $M=2$. A black pixel in row $t$ and column $x$ indicates that $v_t(x) = v_{t-1}(x) + 1$.  The top row is $t = 0$ and $t$ increases from top to bottom. The left column is $x = 1$ and $x$ increases from left to right. }
		\label{fig:sizetwo_dim_1_odometer}
	\end{figure}
	
	See Figure \ref{fig:sizetwo_dim_1_odometer} for a display of the odometer throughout the parallel toppling process when $s^{(d)}_0 \equiv 2d + (d-1)$ in dimension $d=100$. Visually, a contiguous block of decreasing size fires at each step, followed by a ripple of outwards firings. For $t > t_d := (\lceil \sqrt{d-1} \rceil+1)$, the firing block appears to decreases by one every step. In particular, if $a_t$ indexes the right edge of the block at time $t$, then
	$a_1 = d$ and 
	\[
	a_t =
	\begin{cases}
	\lfloor \frac{ d - 1}{t-1} \rfloor & \mbox{ for $t \leq t_d$ } \\
	a_{t-1} - 1 &\mbox{ for $t > t_d $}.
	\end{cases}
	\]
	This leads to a simple formula for $v_t$. 
	
	\begin{prop} \label{prop:base_case}
		For all $d \geq 1$, when $s_0^{(d)} = 2 d + (d-1)$,  the radially reparameterized parallel toppling odometer has the following form. For all $t \geq 1$,
		\begin{equation} \label{eq:continuous_block}
		v_t(x) =
		\begin{cases}
		v_{t-1}(x) + 1  &\mbox{ for $x \leq a_{t}$} \\
		v_{t-1}(x)  &\mbox{ for $a_{t} < x \leq a_{t-1}$}.
		\end{cases} 
		\end{equation}
		And for each $t' < t$ and $a_{t'-1} \geq x > a_{t'}$
		\begin{equation} \label{eq:ripple}
		v_{t}(x) = v_{t-1}(x-1).
		\end{equation}

		
	\end{prop}
	
	\begin{proof}
		We induct on $t$. Since $s_0 \geq 2 d$, $v_1 \equiv 1$.  Suppose \eqref{eq:continuous_block}  and \eqref{eq:ripple} hold for all $t' \leq t$.

		{\bf Step 1: \eqref{eq:continuous_block}} \\
		By strong induction for $t' \leq t$, \eqref{eq:continuous_block} implies $v_t(x) = t$ for $x \leq a_t$. Thus, 
		\[
		s_t(x) = \begin{cases}
		2 d + (d-1) - t x  &\mbox{ for $x < a_t$} \\
		2 d + (d-1) - t x - (d-x) &\mbox{ for $x = a_t$}.
		\end{cases}
		\]
		Let $g(x) := (d-1) - t x$. If $g(x) \geq 0$, $v_{t+1}(x) = v_t(x) + 1$, otherwise $v_{t+1}(x) = v_t(x)$. When $a_{t+1} < x \leq a_t$, $g(x) < 0$. Indeed,  $g(a_{t+1}) \leq (t-1)$ and $g(x+1) - g(x) = - t$.

		As $g$ is increasing in $x$, it remains to check $g(a_{t+1}) \geq 0$ for all $t$.  
		If $x \leq  \frac{ d - 1}{t}$ then $g(x) \geq 0$. 
		If $(t+1) > t_d$, then
		\[
		\frac{d - 1}{t-1}  -  \frac{ d - 1}{t} \leq \frac{d-1}{\sqrt{d-1}\left(\sqrt{d-1}+1\right)} \leq 1,
		\]
		thus 
		\begin{align*}
		a_{t+1} &= a_{t} - 1  \\ 
		&\leq \frac{d-1}{t-1} -1  \\
		&\leq \frac{ d - 1}{t}. 
		\end{align*}

		{\bf Step 2: \eqref{eq:ripple}}
		
		Now, take  $a_{t'-1} \geq x > a_{t'}$ for  $1 \leq t' \leq (t-1)$. If $v_{t}(x-1) = v_{t}(x) + 1$, then by strong induction for $t'\leq t$,
		\eqref{eq:ripple} and \eqref{eq:continuous_block} imply that  $v_{t}(x-2) = v_{t}(x-1) = t'$ and $v_{t}(x) = t'-1 = v_{t}(x+1)$.  Thus, 
		\begin{align*}
		s_t(x) &\geq 2d + (d-1) - (t' - 2) x  \\
		&\geq  2d + (d-1) - (t'-2) a_{t'-1} \\
		&\geq 2d + (d-1) - (d-1)  \\
		&= 2d.
		\end{align*}
		However,  
		\begin{align*}
		s_t(x-1) &= 2 d + (d-1) - t'x - d + x \\
		&= 2 d -1 - (t'-1) x \\
		&\leq 2 d -1,
		\end{align*}
		as $t' \geq 1$.
		If $v_{t}(x-1) = v_{t}(x) = v_{t}(x+1) = (t'-1)$, then
		\begin{align*}
		s_t(x) &\leq 2 d + (d-1) - (t'-1) x  \\
		&< 2 d + (d-1) - (t'-1) (a_{t'}) \\
		&\leq 2d.
		\end{align*} 
		

	\end{proof}

	\section{Odometer regularity when \texorpdfstring{$d = 1$}{d=1}} \label{sec:base_regularity}
		From here onward, suppose $s_0 \equiv 2 d$. We start the inductive proof of Theorem \ref{the_theorem} by establishing some regularity of the odometer in the critical dimension $d=d_0=1$.   In the next section, we inductively use dimensional reduction to show that $d \geq 2$ sandpiles inherit this regularity. This regularity ensures that the dynamics of lower-dimensional sandpiles agree with their higher-dimensional embeddings. 
		
		When reading Section \ref{sec:dim_reduction} below, the reader should observe that whenever Proposition \ref{base_dimension} (or something close to it) holds, dimensional reduction 
		follows. For example, if a version of this result is established in every critical dimension $d_0 \geq 1$,  
		then dimensional reduction follows for all sandpiles of the form $2 d + (d_0 - 1)$ in dimensions $d > d_0$. 
		Proposition \ref{prop:base_case} should be understood as a step in this direction.

	\begin{prop}\label{base_dimension}
		Recall the definition of $\tau_{j}$ from Theorem \ref{the_theorem}. For all $M \geq 2$ and $d \geq 1$, the odometer maintains the following properties throughout the parallel toppling process.
		\begin{description}
			\item[Self-similarity]  For each $1 \leq j \leq M$ and $t \leq \tau_j$ 
			\begin{equation} \label{base_dimension:consistency}
			v_t^{(M)}(\mathbf{x}) =  v_t^{(j)}(\mathbf{x}-(M-j))  \mbox{ for $\mathbf{x} > M-j$}.
			\end{equation}
			\item[Weak facet compatibility] For all $\mathbf{x}_i \geq 2$, $t \geq 1$, $i \geq 0$, $j \geq 0$ and $i+j+1=d$
			\begin{equation} \label{base_dimension:weak_compatability}
			\begin{aligned}
			v_t^{(M)}(\mathbf{x}_i, 1, \mathbf{1}_j) &= v_t^{(M)}(\mathbf{x}_i, 2, \mathbf{1}_j)+ 1  \\
			& \implies \\
			v_{t+1}^{(M)}(\mathbf{x}_i, 1, \mathbf{1}_j) &= v_{t+1}^{(M)}(\mathbf{x}_i, 1, \mathbf{1}_j).
			\end{aligned} 
			\end{equation}
			\item[Strong facet compatibility] For all $\mathbf{x}_i \geq 2$, $j \geq 0$, $i < d$, 
			($t < \tau_M$ and $i \geq 0$) or ($i \geq 1$ and $t \geq \tau_M$)
			\begin{equation} \label{base_dimension:strong_derivative}
			v_t^{(M)}(\mathbf{x}) - v_t^{(M)}(\mathbf{x} + 2 e_{i+1}) \leq 2
			\end{equation}
			and
			\begin{equation} \label{base_dimension:strong_compatability}
			\begin{aligned} 
			v_t^{(M)}(\mathbf{x}_i, 1, \mathbf{1}_j) &= v_t^{(M)}(\mathbf{x}_i, 2, \mathbf{1})+ 1 \\
			& \implies \\
			v_{t+1}^{(M)}(\mathbf{x}_i, 1, \mathbf{1}_j) &= v_{t}^{(M)}(\mathbf{x}_i, 1, \mathbf{1}_j) \\
			v_{t+1}^{(M)}(\mathbf{x}_i, 2, \mathbf{1}_j) &= v_{t}^{(M)}(\mathbf{x}_i, 2, \mathbf{1}_j)+1.
			\end{aligned}
			\end{equation}

			\item[Strong topple control] 
			For all $t \geq \tau_{M-1}$, 
			\begin{equation} \label{base_dimension:strong_topple}
			\sup_{\mathbf{x} \in \mathcal{S}_M} \left(v_{t+2}(\mathbf{x}) - v_{t}(\mathbf{x})\right) \leq 1.
			\end{equation}

		\end{description}
		
	\end{prop}
	
	\begin{proof}[Proof of Proposition \ref{base_dimension} for $d=d_0 = 1$]
		The proof proceeds by induction on $M$ and then on $t$. When $M=2$, 
		\begin{center}
			\begin{tabular}{l|l|l}
				& $v_t^{(M)}(1)$ & $v_t^{(M)}(2)$ \\ \hline
				$t=1$ & 1 & 1  \\ 
				$t=2$ & 2 & 1  \\ 
				$t=3$ & 2 & 2 \\
				$t=4$ & 3 & 2 
			\end{tabular}
		\end{center}
		and $v_1^{(M-1)}(1) = v_2^{(M-1)}(1) = 1$ which verifies the base case. Now, let $M$ be given and note that
		$v_1^{(M)}(x) = 1$ for all $1 \leq x \leq M$ and $v_2^{(M)}(x) = 2$ for $1 \leq x < M$
		and $v_2^{(M)}(M) = 1$. Hence, suppose \eqref{base_dimension:consistency},\eqref{base_dimension:weak_compatability},\eqref{base_dimension:strong_derivative},\eqref{base_dimension:strong_compatability}, hold for $(M-1)$ for all $t \geq 1$ 
		and suppose they hold for $M$ for all $t' \leq (t-1)$. 
		We verify each inductive step. 
		
		\subsubsection*{Self-similarity: \eqref{base_dimension:consistency}}
		By strong induction, it suffices to show $v_t^{(M)}(x) = v_t^{(M-1)}(x-1)$ for $x \geq 2$.
		By Lemma \ref{pt_ind_lemma}, 
		\[
		v_{t}^{(M)}(x) = 1 + \lfloor \frac{v_{t-1}^{(M)}(x+1) + v_{t-1}^{(M)}(x-1)}{2} \rfloor
		\]
		for $x > 1$. 
		Hence, by \eqref{base_dimension:consistency} at $(t-1)$, for $x > 2$, 
		\[
		v_{t}^{(M)}(x) = 1 + \lfloor \frac{v_{t-1}^{(M-1)}(x) + v_{t-1}^{(M-1)}(x-2)}{2} \rfloor = v_t^{(M-1)}(x-1).
		\]
		For $x=1$, we have reflection at the origin, 
		\[
		v_{t}^{(M-1)}(1) = 1 + \lfloor \frac{v_{t-1}^{(M-1)}(2) + v_{t-1}^{(M-1)}(1)}{2} \rfloor = 1 + \lfloor \frac{v_{t-1}^{(M)}(3) + v_{t-1}^{(M)}(2)}{2} \rfloor.
		\]
		Hence, if $v_{t-1}^{(M)}(1) = v_{t-1}^{(M)}(2)$, then $v_{t}^{(M-1)}(1) = v_{t}^{(M)}(2)$. 
		
		When $v_{t-1}^{(M)}(1) = v_{t-1}^{(M)}(2) + 1$, we instead use strong facet compatibility 
		in both layers. If $v_{t-1}^{(M-1)}(1) = v_{t-1}^{(M-1)}(2)$, then $v_{t}^{(M-1)}(1) = v_{t-1}^{(M-1)}(1) + 1$ and we are done, so suppose not. Since sites topple at most once per time step by Lemma \ref{topple_limit}, 
		the odometer must then be,  for some $v \geq 2$:
		\begin{center}
			\begin{tabular}{l | l | l | l }
				& $v_{\cdot}^{(M)}(1)$ & $v_{\cdot}^{(M)}(2)$ & $v_{\cdot}^{(M)}(3)$ \\  \hline
				$t-2$ & $v$ & $v$ & $v-1$    \\
				$t-1$  & $v+1$ & $v$  & $v-1$  \\   
				$t$  & $v+1$ & $v+1$ & $\geq (v-1)$
			\end{tabular}
		\end{center}
		This contradicts strong facet compatibility for $v_t^{(M-1)}(1)$ from $(t-2) \to (t-1)$, which 
		we can use as $t \leq \tau_{M-1}$ and hence $(t-1) < \tau_{M-1}$.
		
		\subsubsection*{Weak facet compatibility: \eqref{base_dimension:weak_compatability}}
		If $v_t(1) = v_t(2) + 1$, then $\Delta v_t(1) = -1$ and so $v_{t+1}(1) = v_{t}(1)$. 
		

		\subsubsection*{Strong facet compatibility: \eqref{base_dimension:strong_derivative} and \eqref{base_dimension:strong_compatability}} 
		
		We use Lemma \ref{general_difference_comparison} to show \eqref{base_dimension:strong_derivative}.
		The function $g(x) = 2 x$ is harmonic in the interior of the interval so it suffices to check 
		the dissipating and reflecting boundaries. We control the dissipating boundary using $t < \tau_M$
		and the reflecting boundary with \eqref{base_dimension:weak_compatability}. 
		
		As $t < \tau_M$, $v_t(M) \leq (M-1)$ and hence by Corollary \ref{cor:derivative_bound}, 
		$v_t(M-1) \leq v_{t}(M) + (M-1) \leq 2 (M-1)$. For the reflecting boundary, i.e, $x=1$, we check  
		\begin{align*}
		\sum_{y \sim x} v_t(y) - \sum_{y' \sim (x+2)} v_t(y) &= (v_t(1) - v_t(2)) + (v_t(2) - v_t(4))  \leq (v_t(1) - v_t(2)) + 4.
		\end{align*}
		If $v_t(1) - v_t(2) = 1$, then $v_{t+1}(1) = v_{t}(1)$ by weak facet compatibility. 
		Otherwise, $\sum_{y \sim x} v_t(y) - \sum_{y' \sim (x+2)} v_t(y) \leq 4$ and we conclude that 
		\[
		v_t(x) - v_t(x + 2) \leq 2 x. 
		\]
		Taking $x=1$, this implies 
		\[
		\Delta v_t(2) \geq -2 v_t(2) + v_t(1) + v_t(1) -2
		\geq 0, 
		\]
		which shows \eqref{base_dimension:strong_compatability}. 
		
		\subsubsection*{Strong topple control: \eqref{base_dimension:strong_topple}}
		By Lemma \ref{topple_limit}, it suffices to show 
		\[ 
		\sup_{x \in \mathcal{S}_M} \left(v_{\tau_{M-1}+2}(x) - v_{\tau_{M-1}}(x)\right) \leq 1
		\]
		First observe that \eqref{base_dimension:consistency} for $v_t^{(M)}$ and \eqref{base_dimension:strong_topple} for $v_t^{(M-1)}$ imply that
		\begin{equation} \label{eq:topple_time_lower_bound}
		\tau_{M-1} \geq \tau_{M-2} + 2. 
		\end{equation}
		Suppose for sake of contradiction that 
		\[
		\left(v_{\tau_{M-1}+2}(x) - v_{\tau_{M-1}}(x)\right) = 2
		\]
		for some $1 \leq x \leq M$. Lemma \ref{topple_limit} then implies
		that some neighbor $y \sim x$ must have toppled twice previously. Pick 
		the maximal such $y$.  We consider three cases for $y$. 
		
		{\bf Case 1: $y \geq 3$} \\
		We first note that $v_{\tau_{M-1}+1}^{(M)}(y) = v_{\tau_{M-1}+1}^{(M-1)}(y-1)$.
		Indeed, by \eqref{base_dimension:consistency}, as $(y-1) \geq 2$, 
		\begin{align*}
		v_{\tau_{M-1}+1}^{(M)}(y) &= 1 + \lfloor \frac{v_{\tau_{M-1}}^{(M)}(y+1) + v_{\tau_{M-1}}^{(M)}(y-1)}{2} \rfloor  \\
		&= 1 + \lfloor \frac{v_{\tau_{M-1}}^{(M-1)}(y) + v_{\tau_{M-1}}^{(M-1)}(y-2)}{2} \rfloor  \\
		&= 	v_{\tau_{M-1}+1}^{(M-1)}(y-1)
		\end{align*}
		Hence, 	
		\[
		2 = v_{\tau_{M-1}+1}^{(M)}(y)-v_{\tau_{M-1}-1}^{(M)}(y) = v_{\tau_{M-1}+1}^{(M-1)}(y-1)-v_{\tau_{M-1}-1}^{(M-1)}(y-1),
		\]
		which contradicts \eqref{base_dimension:strong_topple} for $v_t^{(M-1)}$.

		{\bf Case 2: $y = 2$} \\
		We claim that $v_{\tau_{M-1}+1}^{(M)}(2) = v_{\tau_{M-1}+1}^{(M-1)}(1)$, in which 
		case we can use the argument of Case 1. If not, then $v_{\tau_{M-1}}^{(M)}(2) = v_{\tau_{M-1}}^{(M)}(3)+1$
		but $v_{\tau_{M-1}}^{(M)}(1) = v_{\tau_{M-1}}^{(M)}(2)+1$. This implies that either 
		\[
		v_{\tau_{M-1}-1}^{(M)}(1) = v_{\tau_{M-1}-1}^{(M)}(2)+2
		\]
		or
		\begin{align*}
		v_{\tau_{M-1}-1}^{(M)}(1) &= v_{\tau_{M-1}-1}^{(M)}(2)+1 \\
		 &\mbox{ and }  \\  
		v_{\tau_{M-1}}^{(M)}(1) &= v_{\tau_{M-1}-1}^{(M)}(1)+1
		\end{align*}
		both which contradict weak facet compatibility.

		{\bf Case 3: $y = 1$} \\
		In this case, the odometer near the center must be, for some $v \geq 1$,
			\begin{center}
			\begin{tabular}{l | l | l | l }
				& $v_{\cdot}^{(M)}(1)$ & $v_{\cdot}^{(M)}(2)$ & $v_{\cdot}^{(M)}(3)$ \\  \hline
				$\tau_{M-1}-1$ & $v$ & $v$ & $v$    \\
				$\tau_{M-1}$  & $v+1$ & $v+1$  & $v$  \\   
				$\tau_{M-1}+1$  & $v+2$ & $v+1$ & $\geq (v)$
			\end{tabular}
		\end{center}
		This shows $v^{(M)}_{\tau_{M-1}-2}(2) = v-1$. Indeed, 
		if $v^{(M)}_{\tau_{M-1}-2}(1) = v-1$, then as $\Delta v^{(M)}_{\tau_{M-1}-2}(1) \geq 0$
		 $v^{(M)}_{\tau_{M-1}-2}(2) = v-1$. If $v^{(M)}_{\tau_{M-1}-2}(1) = v$, then $\Delta v^{(M)}_{\tau_{M-1}-2}(1) \leq -1$ and $v^{(M)}_{\tau_{M-1}-2}(2) = v-1$. Hence, 
		 \[
		 v^{(M-1)}_{\tau_{M-1}}(1) = v^{(M)}_{\tau_{M-1}}(2) = v^{(M)}_{\tau_{M-1}-2}(2)+2= v^{(M-1)}_{\tau_{M-1}-2}(2)+2,
		 \]
		 which contradicts \eqref{base_dimension:strong_topple} for $v_t^{(M-1)}$ using \eqref{eq:topple_time_lower_bound}.

	\end{proof}
	
	Note that the comparison principle for sandpiles (see, for example, Proposition 3.3 in \cite{bou2021convergence}) shows
	\[
	v_{\infty}(x) = \frac{1}{2} \left( M(M+1) - x(x-1) \right), 
	\]
	and so $v_{\infty}(x) - v_\infty(x+1) = x$. Hence we must use an assumption like $t < \tau_M$ for strong facet compatibility.

	\section{Odometer regularity and dimensional reduction} \label{sec:dim_reduction}
	We now prove Proposition \ref{base_dimension} for $d \geq 2$ together with dimensional reduction,
	\begin{equation} \label{eq:reduction_gen_d}
	v_t^{(d,M)}(\mathbf{x}_{d-1},1) = v_{t}^{(d-1,M)}(\mathbf{x}_{d-1}),
	\end{equation}
	 by strong induction on $M$, $d$, and $t$. Specifically, given $M$, $d$, and $t$, suppose
	 \[ \eqref{base_dimension:consistency},\eqref{base_dimension:weak_compatability},\eqref{base_dimension:strong_derivative},\eqref{base_dimension:strong_compatability}, \eqref{base_dimension:strong_topple}
	 \]
	 hold for
	\begin{align*}
	v_{t'}^{(d', M')} &\quad \mbox{for all $M' \geq 1$, $t' \geq 1$, $d' < d$}, \\
	v_{t'}^{(d, M')} &\quad \mbox{for all $M' < M$, $t' \geq 1$}, \\
	v_{t'}^{(d, M)} &\quad \mbox{for all $t' < t$}.
	\end{align*}
	We also suppose  \eqref{eq:reduction_gen_d} holds for $v_{t-1}^{(d',M')}$ for all $d' \geq 2$ and $M' \leq M$.
	Indeed, $v_{1}^{(d,M)} \equiv 1$ for all $d \geq 1$ and $M \geq 1$.

	\subsection{Dimensional reduction}
	We start the induction in time by proving dimensional reduction given odometer regularity at $(t-1)$. 
	Let $\mathbf{x}_{d-1}$ be given and pick the smallest $d > i \geq 0$ so that
	$(\mathbf{x}_i, \mathbf{1}_{d-i}) = (\mathbf{x}_{d-1},1)$. 
	By symmetry, 
	\begin{align}
	\notag  \Delta^{(d)} v_{t-1}^{(d)}(\mathbf{x}_{i}, \mathbf{1}_{d-i}) &= -2 d v_{t-1}^{(d)}(\mathbf{x}_{i}, \mathbf{1}_{d-i}) \\  
	\notag  &+  \sum_{\mathbf{y}_{i} \sim \mathbf{x}_{i}} v_{t-1}^{(d)}(\mathbf{y}_{i}, \mathbf{1}_{d-i}) \\
	\notag  &+ (d-i) \left( v_{t-1}^{(d)}(\mathbf{x}_{i}, \mathbf{1}_{d-i})   + v_{t-1}^{(d)}(\mathbf{x}_{i},2, \mathbf{1}_{d-i-1})  \right).
	\end{align}
	
	We consider two cases at time $(t-1)$.

	{\bf Case 1}: $v_{t-1}^{(d)}(\mathbf{x}_{i}, \mathbf{1}_{d-i}) = v_{t-1}^{(d)}(\mathbf{x}_{i},2, \mathbf{1}_{d-i-1})$
	
	 By \eqref{eq:reduction_gen_d} at $(t-1)$, $v_{t-1}^{(d)}(\mathbf{x}_i, \mathbf{1}_{d-i}) = v_{t-1}^{(i+1)}(\mathbf{x}_{i},1)$.
	Thus,
	\begin{align*}
	\Delta^{(d)} v_{t-1}^{(d,M)}(\mathbf{x}_{i}, \mathbf{1}_{d-i}) &= -2 i v_{t-1}^{(i+1)}(\mathbf{x}_{i},1) +  \sum_{\mathbf{y}_{i} \sim \mathbf{x}_{i}} v_{t-1}^{(i+1)}(\mathbf{y}_{i},1)  \\
	&=  \Delta^{(i+1)} v_{t-1}^{(i+1,M)}(\mathbf{x}_{i},1)
	\end{align*}
	which concludes this case as $v_{t}^{(d)} = v_{t-1}^{(d)} + 1(\Delta^{(d)} v_{t-1}^{(d)} \geq 0)$.

	{\bf Case 2: $v_{t-1}^{(d)}(\mathbf{x}_{i}, \mathbf{1}_{d-i}) = v_{t-1}^{(d)}(\mathbf{x}_{i},2, \mathbf{1}_{d-i-1}) + 1$} \\
	If $i \leq (d-2)$, then \eqref{base_dimension:weak_compatability} for $(t-1) \to t$
	for both $v_{t-1}^{(i+1)}$ and $v_{t-1}^{(d)}$ imply that 
	\[
	v_{t-1}^{(i+1)}(\mathbf{x}_i,1) = v_{t-1}^{(d)}(\mathbf{x}_{i}, \mathbf{1}_{d-i})  = v_{t}^{(d)}(\mathbf{x}_{i}, \mathbf{1}_{d-i}) = v_{t}^{(i+1)}(\mathbf{x}_i,1).
	\]
	If $i = (d-1)$, then \eqref{base_dimension:strong_compatability} and \eqref{eq:reduction_gen_d} at $(t-1)$ and $(t-2)$ imply that 
	\[
	v_{t-1}^{(d-1)}(\mathbf{x}_{d-1}) = v_{t-1}^{(d)}(\mathbf{x}_{d-1},1) = v_{t-2}^{(d)}(\mathbf{x}_{d-1},1)+1 = 
	v_{t-2}^{(d-1)}(\mathbf{x}_{d-1})+1.
	\]	
	If $(t-2) \geq \tau_{M-1}$, \eqref{base_dimension:strong_topple} for $v_t^{(i)}$ and $v_t^{(d)}$  imply that
	\[
	v_{t}^{(d-1)}(\mathbf{x}_{d-1}) =	v_{t-1}^{(d-1)}(\mathbf{x}_{d-1}) = v_{t-1}^{(d)}(\mathbf{x}_{d-1},1) = v_{t}^{(d)}(\mathbf{x}_{d-1},1).
	\]
	If $(t-2) < \tau_{M-1}$, then \eqref{base_dimension:consistency} and \eqref{eq:reduction_gen_d} for $v_{t-1}^{(M-1)}$ show
	\[
	v_{t-1}^{(d,M)}(\mathbf{x}_{d-1},2) = v_{t-1}^{(d,M-1)}(\mathbf{x}_{d-1}-1,1) = v_{t-1}^{(d-1,M-1)}(\mathbf{x}_{d-1}-1).
	\]
	Similarly, 
	\[
	v_{t-1}^{(d,M)}(\mathbf{x}_{d-1},1) = v_{t-1}^{(d-1,M)}(\mathbf{x}_{d-1}) = v_{t-1}^{(d-1,M-1)}(\mathbf{x}_{d-1}-1).
	\]
	Therefore, for all $t' \leq \tau_{M-1}$, 
	\begin{equation} \label{eq:reflecting_boundaries}
	v_{t'}^{(d,M)}(\mathbf{x}_{d-1},2) = v_{t'}^{(d,M)}(\mathbf{x}_{d-1},1),
	\end{equation}
	this however contradicts the case we are in.

\subsection{Odometer regularity for \texorpdfstring{$d \geq 2$}{d >= 2}}

We verify each inductive step.

\subsubsection*{Self-similarity: \eqref{base_dimension:consistency}}
As $\eqref{base_dimension:consistency}$ holds for $(M-1)$ at $t$, it suffices to show 
that if $t \leq \tau_{M-1}$ and $\mathbf{x} > 1$,
\begin{equation}\label{eq:self_sim1}
v_t^{(M)}(\mathbf{x}) = v_t^{(M-1)}(\mathbf{x}-1).
\end{equation}
We split verification of this into cases. 

{\bf Case 1: $\mathbf{x} > 2$} \\ 
As \eqref{eq:self_sim1} holds for $(t-1)$, by Lemma \ref{pt_ind_lemma}, 
\begin{align*}
v_{t}^{(M)}(\mathbf{x}) &= \lfloor \frac{ s_0(\mathbf{x}) + \sum_{\mathbf{y} \sim \mathbf{x}} v_{t-1}^{(M)}(\mathbf{y}) }{2 d} \rfloor \\
&= \lfloor \frac{ s_0(\mathbf{x}-1) + \sum_{\mathbf{y} \sim (\mathbf{x}-1)} v_{t-1}^{(M-1)}(\mathbf{y})}{2 d} \rfloor \\
&= v_t^{(M-1)}(\mathbf{x}-1).
\end{align*}

{\bf Case 2:  $\mathbf{x} = (\mathbf{x}_j, \mathbf{2})$ for $\mathbf{x}_j > 2$ and $0 \leq j < d$ } \\
We show that if $\Delta v_{t-1}^{(M)}(\mathbf{x}) \geq 0$, then  $\Delta v_{t-1}^{(M-1)}(\mathbf{x}-1) \geq 0$.
First, decompose the Laplacian into a sum of discrete second differences,
\[
\Delta v_{t-1}^{(M)}(\mathbf{x}) = \sum_{i=1}^d \Delta_{(i)} v_{t-1}^{(M)}(\mathbf{x}),
\]
where 
\[
\Delta_{(i)} v_{t-1}^{(M)}(\mathbf{x}) =  - 2 v_{t-1}^{(M)}(\mathbf{x}) + v_{t-1}^{(M)}(\mathbf{x}+e_i) + v_{t-1}^{(M)}(\mathbf{x}-e_i). 
\]
Observe that \eqref{base_dimension:consistency} at $(t-1)$ implies, 
$\Delta_{(i)} v_{t-1}^{(M)}(\mathbf{x}) = \Delta_{(i)} v_{t-1}^{(M-1)}(\mathbf{x}-1)$
for all $i \leq j$ and for $i > j$, 
\[
\Delta_{(i)} (v_{t-1}^{(M-1)}(\mathbf{x}-1) - v_{t-1}^{(M)}(\mathbf{x})) = v_{t-1}^{(M-1)}(\mathbf{x}-e_i-1) - v_{t-1}^{(M)}(\mathbf{x}-e_i).
\]
By reflectional symmetry, for each $i > j$, 
\[
v_{t-1}^{(M-1)}(\mathbf{x}-e_i-1) = v_{t-1}^{(M-1)}(\mathbf{x}-1) = v_{t-1}^{(M)}(\mathbf{x}),
\]
thus 
\begin{equation} \label{eq:difference}
\Delta_{(i)} (v_{t-1}^{(M-1)}(\mathbf{x}-1) - v_{t-1}^{(M)}(\mathbf{x}))  = v_{t-1}^{(M)}(\mathbf{x}) - v_{t-1}^{(M)}(\mathbf{x}-e_i).
\end{equation}
If $v_{t-1}^{(M)}(\mathbf{x}) = v_{t-1}^{(M)}(\mathbf{x}-e_i)$ for all $i > j$, we are done, so suppose otherwise.

Take $i > j$ where $\eqref{eq:difference} \not = 0$. By \eqref{base_dimension:strong_compatability}, 
\begin{equation} \label{eq:loaded_spring}
v_{t-2}^{(M)}(\mathbf{x}) =  v_{t-1}^{(M)}(\mathbf{x}) =  v_{t-1}^{(M)}(\mathbf{x}-e_i)-1 = v_{t-2}^{(M)}(\mathbf{x}-e_i).
\end{equation}
By \eqref{base_dimension:consistency} and \eqref{base_dimension:strong_compatability} for $v_{t-2}^{(M-1)}$, if
\[
v_{t-1}^{(M-1)}(\mathbf{x}-1) =  v_{t-1}^{(M-1)}(\mathbf{x}-1+e_i)+1
\]
then 
\[
v_{t-1}^{(M)}(\mathbf{x}) = v_{t-1}^{(M-1)}(\mathbf{x}-1) =  v_{t-2}^{(M-1)}(\mathbf{x}-1)+1 = v_{t-2}^{(M)}(\mathbf{x}) + 1
\]
which contradicts \eqref{eq:loaded_spring}. Moreover, 
by \eqref{base_dimension:strong_compatability}
we must have for each neighbor $(\mathbf{y}-e_i) \sim (\mathbf{x} - e_i)$, $v_{t-1}^{(M)}(\mathbf{y}) \geq v_{t-2}^{(M)}(\mathbf{y}-e_i)$. Thus, 
\begin{equation} 
\Delta v_{t-1}^{(M-1)}(\mathbf{x}-1) \geq  \Delta v_{t-2}^{(M)}(\mathbf{x}-e_i) \geq 0.
\end{equation}

\subsubsection*{Strong facet compatibility: \eqref{base_dimension:strong_derivative} and \eqref{base_dimension:strong_compatability}} 
We first use
\begin{equation} \label{eq:two_derivative_bound}
v_t(\mathbf{x}_{i-1}, 1, \mathbf{1}_j) - v_t(\mathbf{x}_{i-1},3, \mathbf{1}_j) \leq 2
\end{equation}
together with the inductive hypotheses 
to show \eqref{base_dimension:strong_compatability}, then we verify \eqref{base_dimension:strong_derivative} below.

Suppose $v_t(\mathbf{x}_{i-1}, 1, \mathbf{1}_j) = v_{t}(\mathbf{x}_{i-1},2,\mathbf{1}_j) + 1$. 
 By \eqref{base_dimension:strong_compatability} at $(\mathbf{x}_i, 1, \mathbf{1}_j)$ from $(t-1) \to t$, $\Delta v_{t-1}(\mathbf{x}_i, 1, \mathbf{1}_j) \geq 0$. Hence, it suffices to show 
 \[
 \Delta v_{t} \left( \mathbf{x}_i, 2, \mathbf{1}_j \right) \geq  \Delta v_{t-1} \left( \mathbf{x}_{i},1, \mathbf{1}_j \right).
 \]
	We use symmetry to decompose each Laplacian;
 \begin{align}
 \notag \Delta v_{t} \left( \mathbf{x}_{i-1},2, \mathbf{1}_{j} \right) &= -2 d v_{t} \left( \mathbf{x}_{i-1},2, \mathbf{1}_{j} \right) \\
 \label{eq:laplace_11} &+ \sum_{j'=1}^{(i-1)} \left( v_{t}(\mathbf{x}_i + e_{j'}, 2, \mathbf{1}_j) + v_{t}(\mathbf{x}_i - e_{j'}, 2, \mathbf{1}_j) \right) \\
  \label{eq:laplace_12} &+ v_t(\mathbf{x}_i, 1, \mathbf{1}_j) + v_t(\mathbf{x}_i, 3, \mathbf{1}_j) \\
  \label{eq:laplace_13} &+ j \left( v_t(\mathbf{x}_i, 2, \mathbf{1}_j) + v_t(\mathbf{x}_i,2,2,\mathbf{1}_{j-1} ) \right)
 \end{align}
 while 
 \begin{align}
\notag \Delta v_{t-1} \left( \mathbf{x}_{i-1},1, \mathbf{1}_{j} \right) &= -2 d v_{t-1} \left( \mathbf{x}_{i-1},1, \mathbf{1}_{j} \right) \\
 \label{eq:laplace_21} &+ \sum_{j'=1}^{(i-1)} \left( v_{t-1}(\mathbf{x}_i + e_{j'}, 1, \mathbf{1}_j) + v_{t-1}(\mathbf{x}_i - e_{j'}, 1, \mathbf{1}_j) \right) \\
 \label{eq:laplace_22} &+ v_{t-1}(\mathbf{x}_i, 1, \mathbf{1}_j) + v_{t-1}(\mathbf{x}_i, 2, \mathbf{1}_j) \\
 \label{eq:laplace_23} &+ j \left( v_{t-1}(\mathbf{x}_i, 1, \mathbf{1}_j) + v_{t-1}(\mathbf{x}_i,2,\mathbf{1}_j) \right).
 \end{align}
  By \eqref{base_dimension:strong_compatability}  from $(t-1) \to t$, $v_{t-1} \left( \mathbf{x}_{i-1},1, \mathbf{1}_{j} \right) = v_{t} \left( \mathbf{x}_{i-1},2, \mathbf{1}_{j} \right)$. Also \eqref{base_dimension:strong_compatability} shows that  
  each $\mathbf{y}_i \sim \mathbf{x}_i$ with $\mathbf{y}_i \geq 2$, $v_{t}(\mathbf{y}_i,2,\mathbf{1}_j) \geq v_{t-1}(\mathbf{y}_i,1,\mathbf{1}_j)$. If $\mathbf{x}_i - e_{j'} \not \geq 2$, $v_{t}(\mathbf{x}_i - e_{j'}, 2, \mathbf{1}_j) = v_{t}(\mathbf{x}_i, 1, \mathbf{1}_j) \geq v_{t-1}(\mathbf{x}_i, 1, \mathbf{1}_j)$.
	This shows that $\eqref{eq:laplace_11} \geq \eqref{eq:laplace_21}$.
  	Next, \eqref{eq:two_derivative_bound} implies 
  \[
  \eqref{eq:laplace_12} \geq 2 v_t(x_i, 1, \mathbf{1}_j) - 2,
  \]
  while \eqref{base_dimension:strong_compatability} from $(t-1) \to t$ 
  implies 
    \[
  \eqref{eq:laplace_22} \leq 2 v_{t-1}(x_i, 1, \mathbf{1}_j),
  \]
  hence $\eqref{eq:laplace_12} \geq \eqref{eq:laplace_22}$.
  Finally, by \eqref{base_dimension:strong_compatability} from $(t-1) \to t$, 
  \[
  v_t(x_i,2,2,\mathbf{1}_{j-1}) \geq v_{t-1}(x_i,2,\mathbf{1}_j)
  \]
  which implies $\eqref{eq:laplace_13} \geq \eqref{eq:laplace_23}$.

	We now verify \eqref{base_dimension:strong_derivative} for different
	regimes of $t$.

{\bf Case 1: $t < \tau_M$, $i \geq 1$} \\
We use Lemma \ref{general_difference_comparison} as in the proof for $d=1$
to show that 
\[
v_t(\mathbf{x}) - v_t(\mathbf{x} + 2 e_i) \leq 2 x_i
\]
for all $d \leq i \leq 1$ and $\mathbf{x} \in \mathcal{S}_M$.  Indeed as $t < \tau_M$
\begin{equation} \label{eq:derivative_bound}
v_{t}(\mathbf{x}) - v_{t}(\mathbf{x} + e_j) \leq x_j
\end{equation}
for all $\mathbf{x} \in \mathcal{S}_M$ and $1 \leq j \leq d$. Hence, 
$v_t(\mathbf{x}) - v_t(\mathbf{x} + 2 e_j) \leq 2 (M-1)$ on $\partial_{2 e_j} \mathcal{S}_M$ for all $1 \leq j \leq d$.
The reflecting boundary is checked in the same way as $d=1$, using weak facet compatibility in higher dimensions.

{\bf Case 2: $t \geq \tau_M$, $i \geq 2$} \\ 
Here we show that
\[
v_t(\mathbf{x}) - v_t(\mathbf{x} + 2 e_i) \leq 2 x_i
\]
for all $d \geq i \geq 2$ and $x_1 \geq 2$. We again use Lemma \ref{general_difference_comparison}
except the region in which we have the derivative bound shrinks and therefore our boundaries change. 
The dissipating boundary gets smaller,  $\mathcal{B}^{(disp)} := \{x \in \mathcal{S}_M : x_j = M-1 \mbox{ for some $2 \leq j \leq d$}\}$
and the reflecting boundary remains the same except for the removal of a single point, $\mathbf{1}$. 
By axis monotonicity, $\sup_{x \in \mathcal{B}^{(disp)}} v_t(x) \leq v_t(M, \mathbf{1}) \leq M$.
Checking the reflective boundary is as in $d=1$ except for the point $(2,\mathbf{1}_{d-1})$.
We show directly that 
\begin{equation} \label{eq:reflecting_point}
v_t(2,\mathbf{1}_{d-1}) \leq v_t(3,2,\mathbf{1}_{d-2}) + 2.
\end{equation}
Suppose for sake of contradiction that $v_t(2,\mathbf{1}_{d-1}) = v_t(3,2,\mathbf{1}_{d-2}) + 2$
and $\Delta v_t(2,\mathbf{1}_{d-1}) \geq 0$ but $\Delta v_t(3,2,\mathbf{1}_{d-1}) < 0$. 
As \eqref{base_dimension:strong_derivative} has been verified for all $x \in \mathcal{S}_M$ other than $(2, \mathbf{1}_{d-1})$, 
\eqref{base_dimension:strong_compatability} holds for $(3,2,\mathbf{1}_{d-1})$ and so 
$v_t(3,2,\mathbf{1}_{d-2}) = v_{t}(3,\mathbf{1}_{d-1})$. 
Then, by definition of the symmetric Laplacian, weak facet compatibility, and axis monotonicity, 
\begin{align*}
\Delta v_t(2,\mathbf{1}_{d-1}) &=  -2d v_t(2, \mathbf{1}_{d-1}) \\
&+ v_t(3, \mathbf{1}_{d-1}) + v_t(1, \mathbf{1}_{d-1}) \\
&+ (d-1) \left( v_t(2,\mathbf{1}_{d-1}) + v_t(2,2,\mathbf{1}_{d-2}) \right) \\
&\leq -2 v_t(2, \mathbf{1}_{d-1}) + v_t(3, \mathbf{1}_{d-1}) + v_t(1, \mathbf{1}_{d-1}) \\
&\leq -1,
\end{align*}
which is a contradiction.

\subsubsection*{Weak facet compatibility: \eqref{base_dimension:weak_compatability}}
The only remaining case is 
\[
v_t^{(d)}(1, \mathbf{1}_j) = v_{t}^{(d)}(2,\mathbf{1}_j) + 1.
\]
By symmetry, 
\[
\Delta v_t^{(d)}(1, \mathbf{1}_j) = d (v_t^{(d)}(2, \mathbf{1}_j) - v_t^{(d)}(1, \mathbf{1}_j)) = -d.
\]

\subsubsection*{Strong topple control: \eqref{base_dimension:strong_topple}}
We use strong topple control established in dimension $(d-1)$.  Suppose for sake of contradiction 
there exists $\mathbf{x} \in \mathcal{S}_M^{(d)}$ with $v_{\tau_{M-1}+2}(\mathbf{x}) - v_{\tau_{M-1}}(\mathbf{x}) = 2$.
Pick $\mathbf{x} = (\mathbf{x}_{d-1}, x)$ so that $x \geq 1$ is minimal. 

{\bf Case 1: $x = 1$}

By dimensional reduction at time $\tau_{M-1}$, $v_{\tau_{M-1}}^{(M,d)}(\mathbf{x}_{d-1},1) = v_{\tau_{M-1}}^{(M,d-1)}(\mathbf{x}_{d-1})$.
By the parabolic least action principle, 
\[
v_{\tau_{M-1}+2}^{(M,d-1)}(\mathbf{x}_{d-1}) \geq v_{\tau_{M-1}+2}^{(M,d)}(\mathbf{x}_{d-1},1),
\]
which contradicts \eqref{base_dimension:strong_topple} for $v_{t}^{(M,d-1)}$.

{\bf Case 2: $x = 2$} \\
By \eqref{eq:reflecting_boundaries}, 
\[
v_{\tau_{M-1}}^{(M,d)}(\mathbf{x}_{d-1},1) = v_{\tau_{M-1}}^{(M,d)}(\mathbf{x}_{d-1},2),
\]
which in turn, by axis monotonicity, implies
$v_{\tau_{M-1}+2}^{(M,d)}(\mathbf{x}_{d-1},1) = v_{\tau_{M-1}}^{(M,d)}(\mathbf{x}_{d-1},1)+2$,
which contradicts the minimality of $x$.

{\bf Case 3: $x \geq 3$}  \\
Some neighbor $\mathbf{y} \sim \mathbf{x}$ must have toppled twice previously. As $x \geq 3$, $\mathbf{y} = (\mathbf{y}_{d-1}, y)$ for $y \geq 2$. The same argument for $d=1$ when $y \geq 2$ then implies, $v_{\tau_{M-1}+1}^{(M,d)}(\mathbf{y}) =  v_{\tau_{M-1}+1}^{(M-1,d)}(\mathbf{y})$ which contradicts \eqref{base_dimension:strong_topple} for $v_{t}^{(M-1,d)}$.

\section*{Declarations}

\begin{itemize}
\item Funding: This research was completed while the author was a graduate student in the Statistics department at the University of Chicago. No external funding was received. 
\item Conflict of interest/Competing interests: Not applicable. 
\item Ethics approval: Not applicable. 
\item Consent to participate: Not applicable. 
\item Consent for publication: Not applicable. 
\item Availability of data and material: Not applicable
\item Code availability: Code to produce the figures in this article is included in the arXiv upload.
\item Authors' contributions: This is a single-author paper. 
\end{itemize}

	\bibliography{bibliography.bib}
	\bibliographystyle{amsalpha}
\end{document}